\RequirePackage{fix-cm}
\documentclass[smallcondensed]{svjour3}     % onecolumn (ditto)
\smartqed  % flush right qed marks, e.g. at end of proof
\usepackage{graphicx}
\usepackage{amsmath}
\usepackage{amssymb}
\usepackage{framed}
\usepackage{enumitem}
\usepackage{multirow}

\allowdisplaybreaks[4]
\DeclareMathOperator*{\tr}{tr}
\DeclareMathOperator*{\argmax}{arg\,max}

\begin{document}

\title{Monotone Mixed Finite Difference Scheme for Monge-Amp\`ere Equation
}

%\titlerunning{Monotone Mixed Scheme for Monge-Amp\`ere Equation}        % if too long for running head

\author{
Yangang Chen
\and
Justin W. L. Wan
\and
Jessey Lin
}

%\authorrunning{Yangang Chen, Justin W. L. Wan} % if too long for running head

\institute{Yangang Chen \at
              Department of Applied Mathematics, University of Waterloo, 200 University Avenue West, Waterloo, ON, N2L 3G1, Canada. \\
              \email{y493chen@uwaterloo.ca}           %  \\
           \and
           Justin W. L. Wan \at
              David R. Cheriton School of Computer Science, University of Waterloo, 200 University Avenue West, Waterloo, ON, N2L 3G1, Canada.
           \and
           Jessey Lin \at
              Centre for Computational Mathematics in Industry and Commerce, University of Waterloo, 200 University Avenue West, Waterloo, ON, N2L 3G1, Canada.
}

\date{Received: date / Accepted: date}

\maketitle

\begin{abstract}
In this paper, we propose a monotone mixed finite difference scheme for solving the two-dimensional Monge-Amp\`ere equation.
In order to accomplish this, we convert the Monge-Amp\`ere equation to an equivalent Hamilton-Jacobi-Bellman (HJB) equation.
Based on the HJB formulation, we apply the standard 7-point stencil discretization, which is second order accurate, to the grid points wherever monotonicity holds, and apply semi-Lagrangian wide stencil discretization elsewhere to ensure monotonicity on the entire computational domain.
By dividing the admissible control set into six regions and optimizing the sub-problem in each region, the computational cost of the optimization problem at each grid point is reduced from $O(M^2)$ to $O(1)$ when the standard 7-point stencil discretization is applied and to $O(M)$ otherwise, where the discretized control set is $M\times M$.
We prove that our numerical scheme satisfies consistency, stability, monotonicity and strong comparison principle, and hence is convergent to the viscosity solution of the Monge-Amp\`ere equation.
In the numerical results, second order convergence rate is achieved when the standard 7-point stencil discretization is applied monotonically on the entire computation domain, and up to order one convergence is achieved otherwise. The proposed mixed scheme yields a smaller discretization error and a faster convergence rate compared to the pure semi-Lagrangian wide stencil scheme.
\end{abstract}

%% Principle
%% ==============

%% 1. The golden standard: Be clear!
%% 2. Focus on the big ideas rather than cooking recipes!
%% 3. Each section should have a clear main idea or objective. Avoid mixing too many different things in each section.
%% 4. Always explain why! It is about telling a story rather than piling up equations
%% 5. Succinct! No repetition!
%% 6. Notations: Stick to standard notations. Avoid introducing too many notations
%% 7. Format: When something looks good, something looks right!
%% 8. Examples and figures are helpful! No need to pursue general formula. Analyze a specific example well (which will avoid both the abstract generality and the repetition)

%% Introduction
%% ==============

\section{Introduction}
\label{sec:intro}

%% Part I: What is the problem?

% What is the Monge-Amp\`ere equation?

The goal of this paper is to compute the numerical solution of the two-dimensional Monge-Amp\`ere equation with Dirichlet boundary condition:
\begin{equation}
\label{eq:MAE}
\renewcommand*{\arraystretch}{1.2}
\begin{array}{rll}
u_{xx} u_{yy} - u_{xy}^2
&
= f,
&
\text{ in } \Omega,
\\
u
&
= g, 
&
\text{ on } \partial\Omega,
\\
u
&
\text{is convex},
&
\end{array}
\end{equation}
where $\Omega$ is a bounded convex open set in $\mathbb{R}^2$, $\partial\Omega$ is its boundary, $\overline{\Omega}=\Omega\cup\partial\Omega$, $u: \overline{\Omega} \to \mathbb{R}$ is the unknown function, and $f: \Omega \to \mathbb{R}$ and $g: \partial\Omega \to \mathbb{R}$ are given functions.

% What is its application?

The Monge-Amp\`ere equation is of great interest due to a wide range of applications, including differential geometry, optimal mass transport (or Monge-Kantorovich) problem, image registration, mesh generation, etc. We direct the interested readers to \cite{caffarelli1999monge} for an extensive review of applications.

% What is the focus (topic) of this paper regarding the Monge-Amp\`ere equation?

The Monge-Amp\`ere equation is a fully nonlinear partial differential equation (PDE), since the left hand side consists of products of the second derivatives. As a result, it may have multiple weak solutions. Among all these weak solutions, we are interested in computing the viscosity solution \cite{crandall1983viscosity,crandall1992user}, since it is often considered the correct one in many practical applications \cite{froese2011convergent}. The viscosity solution of the Monge-Amp\`ere equation is globally convex, while the other solutions may not be convex \cite{froese2011convergent}. We note that a convexity constraint is imposed in the Dirichlet problem (\ref{eq:MAE}) in order to select the viscosity solution and circumvent the issue of multiple weak solutions.

%% Part II: What have people done? (Challenges)

% Category 1: Finite difference methods

Due to the nonlinearity of the Monge-Amp\`ere equation (\ref{eq:MAE}) with the additional convexity constraint, it is challenging to design a numerical scheme that converges to the viscosity solution. Some numerical schemes have been proposed in recent years. One approach is using finite difference methods. Some finite difference schemes, such as \cite{benamou2010two}, use the standard central differencing to discretize $u_{xy}$, and are thus not monotone. The significance of monotonicity is that together with consistency, stability and strong comparison principle, they provide sufficient conditions for a numerical scheme to converge to the viscosity solution \cite{barles1991convergence}.

Very few finite difference schemes that are monotone and thus convergent in the viscosity sense have been proposed. One of the schemes, proposed in \cite{oliker1989numerical}, is to exploit the geometrical interpretation of the Monge-Amp\`ere equation. The grid structure, constrained by the geometry of the equation, is usually not rectangular or triangular. Another scheme, proposed in \cite{oberman2008wide,froese2011convergent}, uses wide stencils to achieve monotonicity. However, in order for the scheme to converge, the number of the stencil points must increase towards infinity when the mesh size $h$ decreases towards 0, thus resulting in high computational costs for solving problems on fine grids. Some improvements on this wide stencil scheme have been proposed. For instance, in \cite{froese2011fast,froese2013filter}, the same authors use hybrid and filtered schemes, both integrating the wide stencil scheme with the more accurate non-monotone central difference scheme in order to improve the accuracy. That being said, %monotonicity may no longer be guaranteed due to the reintroduction of the non-monotone central difference, and 
the issue of infinite stencil points still exists.
Recently, Reference \cite{benamou2016monotone} improves on the previous wide stencil approach so that it is the least nonlocal among all wide stencils of the same family. The number of stencil points does not need to grow to infinity as $h \to 0$, but it still grows and can reach as high as 48.

% Category 2: Finite element methods

Galerkin-type methods have also been developed for solving the Monge-Amp\`ere equation. An immediate challenge is that it is not obvious how to write down the variational formulation of (\ref{eq:MAE}) using the common integration-by-parts approach. The $L^2$ projection methods, proposed in \cite{bohmer2008finite,brenner2011C}, build up the Galerkin-type schemes based on the linearized Monge-Amp\`ere equation. Similar idea can be found in the nonvariational finite element method in \cite{lakkis2013FEM}. In \cite{dean2006numerical}, the authors reformulate the Monge-Amp\`ere equation into an augmented Lagrangian problem or a least-squares problem, which allows the use of mixed finite element methods. The authors in \cite{feng2009vanishing} add an artificial fourth order elliptic differential operator $\epsilon\Delta^2 u$. They show that with this additional term, a variational formulation, and thus a finite element scheme, becomes possible. However, a common issue for these Galerkin-type methods is that convergence to the viscosity solutions for non-regular solutions remains unclear.

%% Part III: What is our contribution?

% Our contribution: Explicitly emphasize the motivations and advantages of our method, which are more important than the method itself.

Our approach, which is distinct from many of the existing methods, is to first convert (\ref{eq:MAE}) into an equivalent Hamilton-Jacobi-Bellman (HJB) equation \cite{krylov1972control,lions1983hamilton}, and then numerically solve the equivalent HJB equation. The application of the HJB formulation in the numerical computation of the Monge-Amp\`ere equation is first investigated by the coauthors of this paper; see the essay \cite{lin2014wide}. Another recent investigation on this approach, \cite{feng2016convergent}, is made public at the completion of our paper. There are some important benefits using the HJB formulation. One is that the differential operator of the HJB equation under fixed control parameters is linear. Another benefit is that the convexity constraint in (\ref{eq:MAE}) is already implicitly incorporated into the HJB differential operator. In other words, there is no need to impose the convexity constraint in the HJB formulation. In addition, many convergent numerical schemes for HJB equations or HJB differential operators have been developed, such as \cite{forsyth2007numerical,howard1960dynamic,debrabant2013semi,ma2014unconditionally,bokanowski2009some,azimzadeh2016weakly,wang2008maximal}. As a result, it is more tractable to design a numerical scheme that converges in the viscosity sense for the equivalent HJB equation than for the Monge-Amp\`ere equation (\ref{eq:MAE}) with the convexity constraint.

Our primary goal is to design a monotone finite difference scheme for the equivalent HJB equation. We note that the cross derivative $u_{xy}$ is still present in the HJB equation, and the standard central differencing or the standard 7-point stencil discretization for $u_{xy}$ may be non-monotone. In order to achieve monotonicity, Reference \cite{feng2016convergent} follows the idea in \cite{debrabant2013semi,ma2014unconditionally} and applies ``semi-Lagrangian scheme" on the entire computational domain, where a local coordinate rotation is performed to remove the cross derivative from the HJB equation, and then central differencing is applied with a stencil length greater than the mesh size $h$, resulting in {\it at most 17 stencil points for any $h$}. In some literature, such semi-Lagrangian scheme is also called wide stencil scheme, which should not to be confused with the wide stencil scheme in \cite{oberman2008wide,froese2011convergent,froese2011fast,froese2013filter} that requires infinity stencil size as $h\to 0$. However, monotonicity is achieved at the expense of large truncation error and slow convergence. In particular, the convergence rate is no better than $O(h)$.

In order to improve the accuracy and meanwhile {\it strictly} maintain monotonicity, our approach is to apply a mixed standard 7-point stencil and semi-Lagrangian wide stencil discretization on the equivalent HJB equation. More specifically, the standard 7-point stencil discretization, which is second order accurate, is applied to discretize $u_{xy}$ at a grid point if monotonicity is fulfilled. Otherwise, the semi-Lagrangian wide stencil scheme, which is less accurate but guaranteed to be monotone, is implemented. We emphasize that our discretization scheme is designed such that consistency, stability, monotonicity and strong comparison principle are fulfilled on the entire computational domain. As a result, our numerical scheme is guaranteed to converge to the viscosity solution of the Monge-Amp\`ere equation \cite{barles1991convergence}. Meanwhile, by maximal use of the standard 7-point stencil discretization, the discretization error of the numerical solution is significantly reduced, compared to the pure semi-Lagrangian wide stencil scheme in \cite{feng2016convergent}. Moreover, our numerical scheme yields a convergence rate of $O(h^2)$ whenever the standard 7-point stencil discretization can be applied monotonically on the entire computation domain, and up to $O(h)$ otherwise. The second order convergence rate in the optimal cases is another significant improvement over the numerical scheme in \cite{feng2016convergent}.

To solve the resulting nonlinear discretized system, one of the most expensive steps is to optimize two control parameters at every grid point. Reference \cite{feng2016convergent} does not discuss the computational cost of the optimization problem. Typically a bilinear search is implemented on an $M\times M$ discretized control set, resulting in $O(M^2)$ computational complexity. We propose an approach that reduces the computational cost for the optimization problem to $O(1)$ whenever the standard 7-point stencil discretization is applied, and at most $O(M)$ otherwise.

Finally, we want to emphasize that our method is the only method that fulfills all the following properties: monotone and thus convergent to the viscosity solution, second order accurate in the optimal cases, and having at most 17 stencil points independent of the mesh size $h$.
None of the references in our paper have the same properties.

%% Part IV: Structure of the paper

To illustrate our numerical scheme, we will briefly review the notion of viscosity solution in Section \ref{sec:viscosity}. In Section \ref{sec:MAEtoHJB}, we will establish the equivalent HJB formulation for the Monge-Amp\`ere equation (\ref{eq:MAE}). In Section \ref{sec:discrete}, we will describe our mixed standard 7-point stencil and semi-Lagrangian wide stencil finite difference discretization for the HJB formulation. Section \ref{sec:policy} solves the nonlinear discretized system using policy iteration, with a detailed discussion on speeding up computation for the optimization of control parameters. Section \ref{sec:converge} proves that our numerical scheme satisfies consistency, stability, monotonicity and strong comparison principle, and thus converges to the viscosity solution of (\ref{eq:MAE}). Section \ref{sec:numerical} shows numerical results. We also demonstrate the discretization error and the rate of convergence for each case. Section \ref{sec:conclusion} is the conclusion.

%% Viscosity Solution of the Monge-Amp\`ere Equation
%% ==============

\section{Viscosity Solution of the Monge-Amp\`ere Equation}
\label{sec:viscosity}

The objective of this paper is to compute the viscosity solution of the Monge-Amp\`ere equation (\ref{eq:MAE}). An overview on the topic of viscosity solution can be found in \cite{crandall1983viscosity,crandall1992user}.

Before defining the viscosity solution of (\ref{eq:MAE}), we rewrite (\ref{eq:MAE}) as
\begin{equation}
\label{eq:MAE2}
\renewcommand*{\arraystretch}{1.2}
\begin{array}{l}
\mathcal{F}
\left(
\mathbf{x}, u(\mathbf{x}), D^2 u(\mathbf{x})
\right)
\equiv 
\left\{
\begin{array}{ll}
-\det
\left[
D^2 u(\mathbf{x})
\right]
+ f (\mathbf{x}),
&
\mathbf{x} \in \Omega,
\\
u(\mathbf{x}) - g(\mathbf{x}),
&
\mathbf{x} \in \partial\Omega,
\end{array}
\right.
= 0,
\\
u \text{ is convex }
\quad \Rightarrow \quad 
D^2 u(\mathbf{x}) \text{ is positive semi-definite},
\end{array}
\end{equation}
where $\mathbf{x}=(x,y)\in\overline{\Omega}$, and $D^2 u$ is the Hessian matrix of $u$.

To introduce the notion of viscosity solution, we define the upper (respectively lower) semi-continuous envelope of a function $z: C \to \mathbb{R}$ on a closed set $C$ as
\begin{equation}
\label{eq:envelope}
z^*(x)
\equiv \displaystyle\limsup_{\substack{y\to x, \, y\in C}} z(y)
\quad
\left(
\;
\text{respectively }
z_*(x)
\equiv \displaystyle\liminf_{\substack{y\to x, \, y\in C}} z(y)
\;
\right).
\end{equation}

\begin{definition}[Viscosity solution]
\label{def:viscosity}
A convex upper (respectively lower) semi-continuous function $u:\overline{\Omega}\to \mathbb{R}$ is a viscosity subsolution (respectively supersolution) of the Monge-Amp\`ere equation
$\mathcal{F}
\left(
\mathbf{x}, u(\mathbf{x}), D^2 u(\mathbf{x})
\right) = 0$,
if for all the test functions $\varphi(\mathbf{x})\in C^2(\overline{\Omega})$ and all $\mathbf{x}\in\overline{\Omega}$, such that $u^*-\varphi$ (respectively $u_*-\varphi$) has a local maximum (respectively minimum) at $\mathbf{x}$, we have
\begin{equation}
\mathcal{F}_* (\mathbf{x}, u^*(\mathbf{x}), D^2\varphi(\mathbf{x})) \leq 0
\quad
\left(
\;
\text{respectively }
\mathcal{F}^* (\mathbf{x}, u_*(\mathbf{x}), D^2\varphi(\mathbf{x})) \geq 0
\;
\right).
\end{equation}
Furthermore, the function $u$ is a viscosity solution if it is both a viscosity sub-solution and super-solution.
\end{definition}

We note that the convexity of $u$ (or equivalently, $D^2 u$ being positive semi-definite, $\det(D^2 u) = f\geq 0$) already implies that the differential operator of (\ref{eq:MAE2}) is degenerate elliptic. Furthermore, degenerate ellipticity, plus $\overline{\Omega}$ being bounded and convex, ensures the existence and uniqueness of the viscosity solution of (\ref{eq:MAE2}). See \cite{crandall1992user,gutierrez2012monge} for details.

%% HJB Formulation of the Monge-Amp\`ere Equation
%% ==============

\section{HJB Formulation of the Monge-Amp\`ere Equation}
\label{sec:MAEtoHJB}

%% \subsubsection*{Converting Monge-Amp\`ere Equation to Hamilton-Jacobi-Bellman Equation}

Since the Monge-Amp\`ere equation (\ref{eq:MAE2}) is nonlinear, it is challenging to design a finite difference scheme that converges to the viscosity solution. Our approach is to convert the Monge-Amp\`ere equation into an equivalent HJB equation. The equivalence of the two PDEs is first established in \cite{krylov1972control} and \cite{lions1983hamilton} for classical solutions. Recently, Reference \cite{feng2016convergent} extends the equivalence to the setting of viscosity solutions. Here we state the equivalence of the two PDEs as the following theorem:

\begin{theorem}
\label{thm:MAEtoHJB}
Let $\Omega$ be a convex open set in $\mathbb{R}^2$. Let $f\in C(\Omega)$ be a non-negative function. Let a convex function $u$ be the viscosity solution of the following HJB equation,
\begin{equation}
\label{eq:convertHJB}
\displaystyle\max_{A(\mathbf{x}) \in S_1^+} \left\{
- \tr \left[
A (\mathbf{x}) D^2 u (\mathbf{x})
\right]
+ 2 \sqrt{
\det(A(\mathbf{x})) \,
f (\mathbf{x})
}
\right\} = 0,
\end{equation}
where
$S_1^+
\equiv
\{ A\in \mathbb{R}^{2\times 2}: \, A \text{ is positive semi-definite, } A^T = A, \, \tr(A)=1 \}$
and $A(\mathbf{x})\in S_1^+$ is the control at point $\mathbf{x}$. Then $u$ is the viscosity solution of the Monge-Amp\`ere equation (\ref{eq:MAE2}).
\end{theorem}

\begin{proof}
We refer interested readers to the proof in \cite{smearshamilton} when $u$ is a classical solution, and the proof in \cite{feng2016convergent} for the extension to the viscosity solution.
\hfill
\end{proof}

We notice that due to the positive semi-definite property of the matrix $A(\mathbf{x})$, it can be diagonalized by an order-two orthogonal matrix. More specifically, $A(\mathbf{x})\in S_1^+$ can be parametrized as follows:
\begin{equation}
\label{eq:parametrize}
\renewcommand*{\arraystretch}{1}
\begin{array}{r}
A(\mathbf{x}) = 
\left(
\renewcommand*{\arraystretch}{1}
\begin{matrix}
\cos\theta(\mathbf{x}) & \sin\theta(\mathbf{x}) \\
-\sin\theta(\mathbf{x}) & \cos\theta(\mathbf{x}) \\
\end{matrix}
\right)
\left(
\renewcommand*{\arraystretch}{1}
\begin{matrix}
a(\mathbf{x}) & 0 \\
0 & 1-a(\mathbf{x}) \\
\end{matrix}
\right)
\left(
\renewcommand*{\arraystretch}{1}
\begin{matrix}
\cos\theta(\mathbf{x}) & -\sin\theta(\mathbf{x}) \\
\sin\theta(\mathbf{x}) & \cos\theta(\mathbf{x}) \\
\end{matrix}
\right),
\\
a(\mathbf{x}) \in [0,1], \;
\theta(\mathbf{x}) \in [-\pi,\pi).
\end{array}
\end{equation}
This parametrization gives rise to the following HJB equation, which we aim at solving.
\begin{corollary}
Under the parametrization (\ref{eq:parametrize}), the HJB equation (\ref{eq:convertHJB}) becomes
\begin{equation}
\label{eq:HJB1}
\renewcommand*{\arraystretch}{1}
\begin{array}{rl}
\displaystyle\max_{(a(\mathbf{x}),\theta(\mathbf{x}))\in \Gamma}
&
\left\{
- \alpha_{11} (a(\mathbf{x}),\theta(\mathbf{x}))
u_{xx}(\mathbf{x})
- 2 \alpha_{12} (a(\mathbf{x}),\theta(\mathbf{x}))
u_{xy}(\mathbf{x})
\right.
\\
&
\left.
- \alpha_{22} (a(\mathbf{x}),\theta(\mathbf{x}))
u_{yy}(\mathbf{x})
+ 2\sqrt{a(\mathbf{x})(1-a(\mathbf{x}))f(\mathbf{x})}
\right\}
= 0,
\end{array}
\end{equation}
where $(a(\mathbf{x}),\theta(\mathbf{x}))$ is the pair of controls at point $\mathbf{x}$, $\Gamma = [0,1]\times\left[-\frac{\pi}{4},\frac{\pi}{4}\right)$ is the set of admissible controls\footnote{Although (\ref{eq:parametrize}) defines the admissible control set to be in the range of $[0,1]\times[-\pi,\pi)$, the optimal control pair $(a^*,\theta^*)$ that maximizes (\ref{eq:HJB1}) may not be unique in $[0,1]\times[-\pi,\pi)$. We notice that since $\mathcal{L}_{a,\theta} \, u = \mathcal{L}_{a,\theta+\pi} \, u$, and $\mathcal{L}_{a,\theta} \, u = \mathcal{L}_{1-a,\theta+\frac{\pi}{2}} \, u$, the admissible control set $\Gamma$ can be reduced to $[0,1]\times[-\frac{\pi}{4},\frac{\pi}{4})$. Such removal of the redundancy of $\Gamma$ ensures that the optimal control pair $(a^*,\theta^*)$ is unique in $\Gamma$, except when $a^*=\frac{1}{2}$ or when $f=0$.}
, and the coefficients are
\begin{equation}
\label{eq:HJBcoeff}
\renewcommand*{\arraystretch}{1.3}
\begin{array}{rl}
\alpha_{11}(a(\mathbf{x}),\theta(\mathbf{x})) &
= \frac{1}{2} [ 1-(1-2a(\mathbf{x}))\cos 2\theta(\mathbf{x}) ],
\\
\alpha_{22}(a(\mathbf{x}),\theta(\mathbf{x})) &
= \frac{1}{2} [ 1+(1-2a(\mathbf{x}))\cos 2\theta(\mathbf{x}) ],
\\
\alpha_{12}(a(\mathbf{x}),\theta(\mathbf{x})) &
= \frac{1}{2} (1-2a(\mathbf{x}))\sin 2\theta(\mathbf{x}).
\end{array}
\end{equation}
\end{corollary}

For convenience, we rewrite the HJB equation (\ref{eq:HJB1}) as
\begin{equation}
\label{eq:HJB2}
\mathcal{F}
\left(
\mathbf{x}, u(\mathbf{x}), D^2 u(\mathbf{x})
\right)
\equiv 
\displaystyle\max_{(a(\mathbf{x}),\theta(\mathbf{x}))\in \Gamma}
\mathcal{L}_{a(\mathbf{x}),\theta(\mathbf{x})} \, u(\mathbf{x})
= 0,
\end{equation}
where the differential operator of the HJB equation is given by
\begin{equation}
\label{eq:HJBdiff}
\mathcal{L}_{a,\theta} \, u
\equiv
- \alpha_{11}(a,\theta) u_{xx}
- 2 \alpha_{12}(a,\theta) u_{xy}
- \alpha_{22}(a,\theta) u_{yy}
+ 2\sqrt{a(1-a)f}.
\end{equation}
We note that since the HJB equation (\ref{eq:HJB2})-(\ref{eq:HJBdiff}) and the Monge-Amp\`ere equation (\ref{eq:MAE2}) are mathematically equivalent, we still use the notation
$\mathcal{F}
\left(
\mathbf{x}, u(\mathbf{x}), D^2 u(\mathbf{x})
\right)$
to denote the HJB equation.

%% \subsubsection{Why using HJB formulation?}

The HJB formulation introduces some favorable properties over the Monge-Amp\`ere equation (\ref{eq:MAE2}). We first notice that in the equivalent HJB equation (\ref{eq:convertHJB}) or (\ref{eq:HJB1}), the convexity constraint of the Monge-Amp\`ere equation disappears. Indeed, the convexity constraint is implicitly enforced in the HJB formulation.
The reason is that the proof of Theorem \ref{thm:MAEtoHJB}, where the Monge-Amp\`ere equation is converted to the HJB equation, has already taken into account that $u$ is a convex function.
We remark that the convexity constraint poses a major difficulty in designing a convergent numerical scheme for Monge-Amp\`ere equation; see \cite{froese2011convergent} for a discussion. However, in the HJB formulation, there is no need to explicitly impose the convexity constraint any more, which makes the numerical computation more manageable.

Another useful property of the HJB equation (\ref{eq:HJB2})-(\ref{eq:HJBdiff}) is that for a fixed given control pair $(a,\theta)$, the differential operator $\mathcal{L}_{a,\theta} \, u$ is linear. We note, however, that the HJB equation itself is still nonlinear, since the maximization depends on $u$. Unlike (\ref{eq:MAE2}), the linear differential operator $\mathcal{L}_{a,\theta} \, u$ does not contain products of the second derivatives. The linearity of $\mathcal{L}_{a,\theta} \, u$ allows us to develop finite difference schemes based on numerical methods for linear PDEs.

Considering these advantages of the HJB formulation, our approach is to solve the HJB equation (\ref{eq:HJB1}) instead of the Monge-Amp\`ere equation (\ref{eq:MAE2}).

%% Mixed Finite Difference Discretization
%% ==============

\section{Mixed Finite Difference Discretization}
\label{sec:discrete}

In this section, we will construct a monotone finite difference discretization for the HJB equation (\ref{eq:HJB1}). Monotonicity is a desirable property, since \cite{barles1991convergence} has proved that monotonicity is one of the sufficient conditions for a numerical scheme to converge to the viscosity solution.

To set up notation, let us consider an $N\times N$ square grid
$
\left\{ \mathbf{x}_{i,j}=(x_i,y_j) \right\}
$,
where $\mathbf{x}_{i,j}\in\Omega$ when $i,j=1,\cdots,N$, and $\mathbf{x}_{i,j}\in\partial\Omega$ when $i,j=0$ or $N+1$. Also, let $h$ be the mesh size and let $u_{i,j}$, $a_{i,j}$, $\theta_{i,j}$ and $f_{i,j}$ be the grid functions of $u(\mathbf{x}_{i,j})$, $a(\mathbf{x}_{i,j})$, $\theta(\mathbf{x}_{i,j})$ and $f(\mathbf{x}_{i,j})$, respectively. Our goal is to solve the set of the unknowns $\{u_{i,j} \,|\, 1 \leq i \leq N, 1 \leq j \leq N \}$.

\subsection{Standard 7-point stencil discretization}
\label{subsec:narrow}

Consider discretizing the HJB equation (\ref{eq:HJB1}) at a grid point $\mathbf{x}_{i,j}$. We can use the standard central differencing to approximate $u_{xx}(\mathbf{x}_{i,j})$  and $u_{yy}(\mathbf{x}_{i,j})$ as follows:
\begin{equation}
\label{eq:derivative_narrow0}
(\delta_{xx} u)_{i,j}
\equiv
\frac{u_{i+1,j}-2u_{i,j}+u_{i-1,j}}{h^2},
\quad
(\delta_{yy} u)_{i,j}
\equiv
\frac{u_{i,j+1}-2u_{i,j}+u_{i,j-1}}{h^2}.
\end{equation}
It can be shown that the standard 7-point stencil discretization for $u_{xy}(\mathbf{x}_{i,j})$ can lead to a monotone scheme in the following two cases:

\underline{Case 1.} When the coefficients $\alpha_{11}$, $\alpha_{22}$ and $\alpha_{12}$ in (\ref{eq:HJBcoeff}) satisfy
\begin{equation}
\label{eq:narrow_cond1}
\renewcommand*{\arraystretch}{1}
\begin{array}{r}
\alpha_{11}(a_{i,j},\theta_{i,j}) \geq |\alpha_{12}(a_{i,j},\theta_{i,j})|, \;
\alpha_{22}(a_{i,j},\theta_{i,j}) \geq |\alpha_{12}(a_{i,j},\theta_{i,j})|, \;
\\
\text{and }
\alpha_{12}(a_{i,j},\theta_{i,j})\geq 0
\text{ at the grid point } \mathbf{x}_{i,j},
\end{array}
\end{equation}
we approximate $u_{xy}(\mathbf{x}_{i,j})$ using
\begin{equation}
\label{eq:derivative_narrow1}
(\delta_{xy}^{[1]} u)_{i,j}
\equiv
\frac
{2u_{i,j}+u_{i+1,j+1}+u_{i-1,j-1}
-u_{i+1,j}-u_{i-1,j}-u_{i,j+1}-u_{i,j-1}}
{2h^2}.
\end{equation}

\underline{Case 2.} When the coefficients $\alpha_{11}$, $\alpha_{22}$ and $\alpha_{12}$ in (\ref{eq:HJBcoeff}) satisfy
\begin{equation}
\label{eq:narrow_cond2}
\renewcommand*{\arraystretch}{1}
\begin{array}{r}
\alpha_{11}(a_{i,j},\theta_{i,j}) \geq |\alpha_{12}(a_{i,j},\theta_{i,j})|, \;
\alpha_{22}(a_{i,j},\theta_{i,j}) \geq |\alpha_{12}(a_{i,j},\theta_{i,j})|, \;
\\
\text{and }
\alpha_{12}(a_{i,j},\theta_{i,j})\leq 0
\text{ at the grid point } \mathbf{x}_{i,j},
\end{array}
\end{equation}
we approximate $u_{xy}(\mathbf{x}_{i,j})$ using
\begin{equation}
\label{eq:derivative_narrow2}
(\delta_{xy}^{[2]} u)_{i,j}
\equiv
\frac
{-2u_{i,j}-u_{i+1,j-1}-u_{i-1,j+1}
+u_{i+1,j}+u_{i-1,j}+u_{i,j+1}+u_{i,j-1}}
{2h^2}.
\end{equation}

%We remark that Condition (\ref{eq:narrow_cond1}) and Condition (\ref{eq:narrow_cond2}) depend on the coefficients $\alpha_{11}$, $\alpha_{22}$ and $\alpha_{12}$, which in turn depend on the pair of controls $(a_{i,j},\theta_{i,j})$.

\subsection{Semi-Lagrangian wide stencil discretization}
\label{subsec:wide}

%% \subsubsection*{Eliminating the cross derivative}

However, if neither (\ref{eq:narrow_cond1}) nor (\ref{eq:narrow_cond2}) is fulfilled at the grid point $\mathbf{x}_{i,j}$, then it is unclear how to directly discretize the cross derivative $u_{xy}(\mathbf{x}_{i,j})$ in (\ref{eq:HJB1}) monotonically. Our approach, following \cite{debrabant2013semi} and \cite{ma2014unconditionally}, is to eliminate the cross derivative $u_{xy}(\mathbf{x}_{i,j})$ by a local coordinate transformation. Let $\{(\mathbf{e}_z)_{i,j},(\mathbf{e}_w)_{i,j}\}$ be a local orthogonal basis which is obtained by a rotation of the standard axes $\{(\mathbf{e}_x)_{i,j},(\mathbf{e}_y)_{i,j}\}$ at an angle $\phi_{i,j}$; see Figure \ref{fig:widestencil_all} (left). If the rotation angle is chosen as
$
\phi_{i,j}
= \frac{1}{2} \arctan
\frac{
2\alpha_{12}\left(
a_{i,j},\theta_{i,j}
\right)
}
{
\alpha_{11}\left(
a_{i,j},\theta_{i,j}
\right)
- \alpha_{22}\left(
a_{i,j},\theta_{i,j}
\right)
}
= -\theta_{i,j}
$,
then the cross derivative vanishes under the basis $\{(\mathbf{e}_z)_{i,j},(\mathbf{e}_w)_{i,j}\}$. By straightforward algebra, one can show that (\ref{eq:HJB1}) becomes
\begin{equation}
\label{eq:disc_wide1}
\displaystyle\max_{(a_{i,j},\theta_{i,j})\in \Gamma}
\left\{
- a_{i,j} \,
u_{zz} (\mathbf{x}_{i,j})
- \left( 1-a_{i,j} \right) \,
u_{ww} (\mathbf{x}_{i,j})
+ 2\sqrt{
a_{i,j} \left( 1-a_{i,j} \right) f_{i,j}
}
\right\}
= 0.
\end{equation}
Here $u_{zz}(\mathbf{x}_{i,j})$ and $u_{ww}(\mathbf{x}_{i,j})$ are the directional derivatives along the basis $(\mathbf{e}_z)_{i,j}$ and $(\mathbf{e}_w)_{i,j}$, which depend on the rotation $\theta_{i,j}$.

%% \subsubsection*{Semi-Lagrangian wide stencil discretization and bilinear interpolation}

\begin{figure}[t!]
\begin{center}
\includegraphics[scale=0.22]{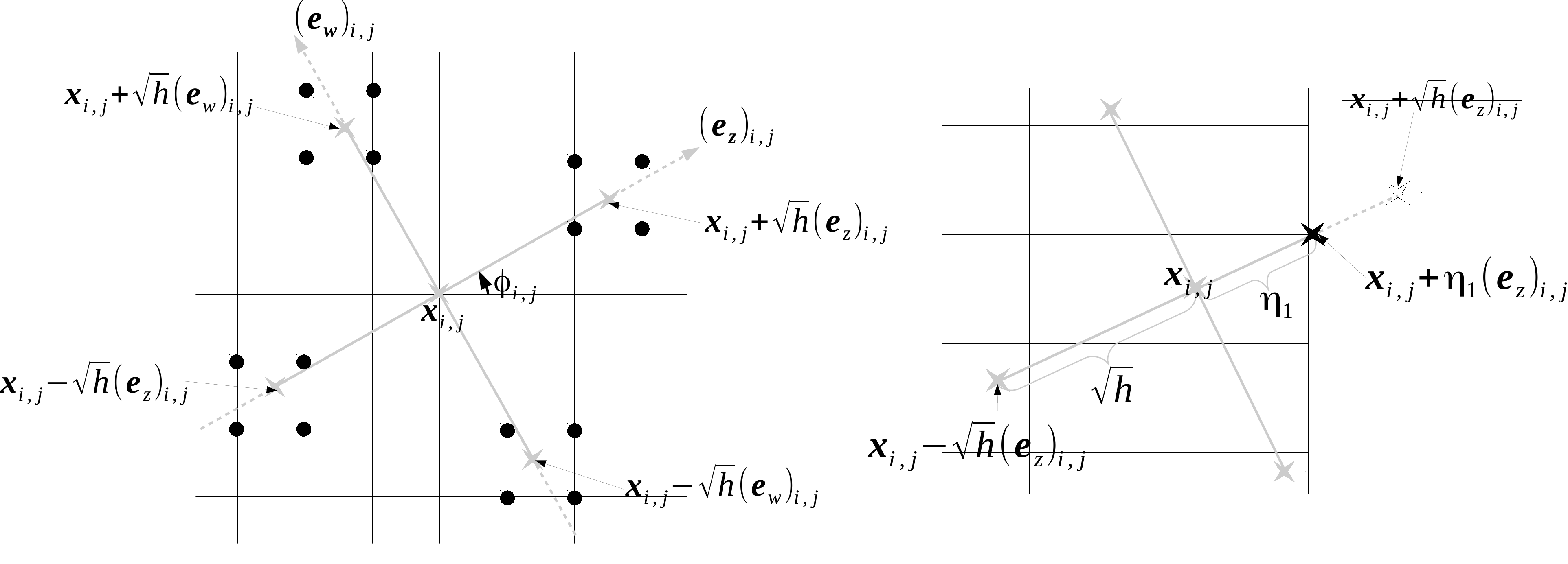}
\end{center}
\caption{\label{fig:widestencil_all}
(left) Local coordinate rotation at the grid point $\mathbf{x}_{i,j}$, and semi-Lagrangian wide stencil discretization of $u_{zz}(\mathbf{x}_{i,j})$ and $u_{ww}(\mathbf{x}_{i,j})$ under the rotation. The rotation angle is $\phi_{i,j}$, counter-clockwise. The grey dashed lines are the orthogonal axis $\{(\mathbf{e}_z)_{i,j}, (\mathbf{e}_w)_{i,j}\}$. The stencil length is $\sqrt{h}$ ($\sqrt{h}>h$). The grey stars are the stencil points $\mathbf{x}_{i,j}\pm\sqrt{h}(\mathbf{e}_z)_{i,j}$ and $\mathbf{x}_{i,j}\pm\sqrt{h}(\mathbf{e}_w)_{i,j}$. The unknowns at these stencil points are approximated by the bilinear interpolation from the neighboring points (black dots). Standard central differencing associated with this wide stencil is applied to approximate $u_{zz}(\mathbf{x}_{i,j})$ and $u_{ww}(\mathbf{x}_{i,j})$.
\quad
(right) Semi-Lagrangian wide stencil discretization near the boundary. One of the wide stencil points $\mathbf{x}_{i,j}+\sqrt{h}(\mathbf{e}_z)_{i,j}$ falls outside $\overline{\Omega}$ (hollow star). The wide stencil is truncated and the stencil point is relocated to the point $\mathbf{x}_{i,j}+\eta_1(\mathbf{e}_z)_{i,j}\in\partial\Omega$ (black star). The corresponding stencil length has shrunk from $\sqrt{h}$ to $\eta_1$.
}
\end{figure}

We may consider the finite difference discretization of (\ref{eq:disc_wide1}) by applying the standard central differencing to $u_{zz} (\mathbf{x}_{i,j})$ and $u_{ww} (\mathbf{x}_{i,j})$. For instance, we approximate
$u_{zz} (\mathbf{x}_{i,j})$
by
$\frac{1}{h^2}\left[
u(\mathbf{x}_{i,j}+h(\mathbf{e}_z)_{i,j})
-2u_{i,j}
+u(\mathbf{x}_{i,j}-h(\mathbf{e}_z)_{i,j})
\right]$.
However, since the stencil is rotated, the stencil points $\mathbf{x}_{i,j}\pm h(\mathbf{e}_z)_{i,j}$ may no longer coincide with any grid points. In such cases, bilinear interpolation from the neighboring grid points can be used to approximate $u(\mathbf{x}_{i,j}\pm h(\mathbf{e}_z)_{i,j})$. However, a consequence of the bilinear interpolation is that the truncation error of this central difference approximation becomes $O(1)$ if the stencil length is $h$. In order to maintain consistency, we choose the stencil length $\sqrt{h}$, which yields $O(h)$ truncation error. Note that when $h$ is small, $\sqrt{h}>h$, which means the stencil length appears to be wide. The details of the discretization is explained in Figure \ref{fig:widestencil_all} (left). As a result, the finite difference discretization for $u_{zz} (\mathbf{x}_{i,j})$ and $u_{ww} (\mathbf{x}_{i,j})$ is given by
\begin{align}
\renewcommand*{\arraystretch}{2.4}
\label{eq:derivative_wide1}
&
(\delta_{zz} u)_{i,j}
\equiv
\frac{
\left.
\mathcal{I}_h u
\right|_{\mathbf{x}_{i,j}+\sqrt{h}(\mathbf{e}_z)_{i,j}}
-2u_{i,j}
+\left.
\mathcal{I}_h u
\right|_{\mathbf{x}_{i,j}-\sqrt{h}(\mathbf{e}_z)_{i,j}}
}
{h},
\\
\label{eq:derivative_wide2}
&
(\delta_{ww} u)_{i,j}
\equiv
\frac{
\left.
\mathcal{I}_h u
\right|_{\mathbf{x}_{i,j}+\sqrt{h}(\mathbf{e}_w)_{i,j}}
-2u_{i,j}
+\left.
\mathcal{I}_h u
\right|_{\mathbf{x}_{i,j}-\sqrt{h}(\mathbf{e}_w)_{i,j}}
}
{h},
\end{align}
where we have used the stencil length $\sqrt{h}$, and used bilinear interpolation to approximate the unknown values at the stencil points $\mathbf{x}_{i,j}\pm\sqrt{h}(\mathbf{e}_z)_{i,j}$ and $\mathbf{x}_{i,j}\pm\sqrt{h}(\mathbf{e}_w)_{i,j}$, denoted as
$\left.
\mathcal{I}_h u
\right|_{\mathbf{x}_{i,j}\pm\sqrt{h}(\mathbf{e}_z)_{i,j}}$
and
$\left.
\mathcal{I}_h u
\right|_{\mathbf{x}_{i,j}\pm\sqrt{h}(\mathbf{e}_w)_{i,j}}$.
Such discretization scheme is called semi-Lagrangian wide stencil discretization \cite{debrabant2013semi,ma2014unconditionally}.

%% \subsubsection*{Discretization near the boundary}

If we apply the semi-Lagrangian wide stencil discretization at a grid point $\mathbf{x}_{i,j}$ that is close to the boundary, some of its associated stencil points may fall outside the computational domain $\overline{\Omega}$. In such case, our solution is to shrink the corresponding stencil length(s) such that the stencil point(s) are relocated onto the boundary $\partial\Omega$. Without loss of generality, we analyze one scenario; see Figure \ref{fig:widestencil_all} (right). Let us assume that $\mathbf{x}_{i,j}+\sqrt{h}(\mathbf{e}_z)_{i,j}$ falls outside $\overline{\Omega}$. We truncate the corresponding stencil length from $\sqrt{h}$ to $\eta_1$ along the $\mathbf{e}_z$ axis, such that the stencil point is relocated to $\mathbf{x}_{i,j}+\eta_1 (\mathbf{e}_z)_{i,j}\in\partial\Omega$. Since $\eta_1\neq\sqrt{h}$, the finite difference approximation for $u_{zz}(\mathbf{x}_{i,j})$ in (\ref{eq:derivative_wide1}) is replaced by
\begin{equation}
\label{eq:derivative_wide1_bc1}
(\delta_{zz} u)_{i,j}
\equiv
\frac{
\frac{
g(\mathbf{x}_{i,j}+\eta_1 (\mathbf{e}_z)_{i,j})
-u_{i,j}
}
{\eta_1}
- \frac{
u_{i,j}
- \left.
\mathcal{I}_h u
\right|_{\mathbf{x}_{i,j}-\sqrt{h}(\mathbf{e}_z)_{i,j}}
}
{\sqrt{h}}
}
{
\frac{\eta_1+\sqrt{h}}{2}
},
\end{equation}
where we have used the Dirichlet boundary condition of (\ref{eq:MAE}): $u(\mathbf{x}_{i,j}+\eta_1 (\mathbf{e}_z)_{i,j}) = g(\mathbf{x}_{i,j}+\eta_1 (\mathbf{e}_z)_{i,j})$. We note that such procedure can be used whenever $\mathbf{x}_{i,j}$ is close to the boundary and a truncation of stencil is needed.

\subsection{Mixed discretization}
\label{subsec:mixed}

Section \ref{subsec:narrow} and \ref{subsec:wide} describe the standard 7-point stencil and semi-Lagrangian wide stencil finite difference discretization for the HJB equation (\ref{eq:HJB1}). The advantage of the semi-Lagrangian wide stencil discretization is that it is unconditionally monotone. Reference \cite{feng2016convergent} applies the semi-Lagrangian wide stencil discretization at every grid point. However, it is only first order accurate, while the standard 7-point stencil discretization is second order accurate, as will be proved in Section \ref{sec:converge}. In order to combine the advantages of both discretization schemes, we will only apply the semi-Lagrangian wide stencil discretization at the grid points where neither (\ref{eq:narrow_cond1}) nor (\ref{eq:narrow_cond2}) is satisfied. For the other grid points where either (\ref{eq:narrow_cond1}) or (\ref{eq:narrow_cond2}) is fulfilled, we will apply the standard 7-point stencil discretization. The purpose is to strictly maintain monotonicity at every grid point and meanwhile to make the numerical scheme as accurate as possible. As a result, the discrete equation at each grid point $\mathbf{x}_{i,j}$ is given by the following mixed scheme:

\underline{Standard 7-point stencil discretization.} When the control pair $(a_{i,j},\theta_{i,j})$ satisfies Condition (\ref{eq:narrow_cond1}) or (\ref{eq:narrow_cond2}), the discrete equation is given by
\begin{equation}
\label{eq:disc_narrow}
\renewcommand*{\arraystretch}{1}
\begin{array}{rl}
\displaystyle\max_{(a_{i,j},\theta_{i,j})\in \Gamma}
\left\{
- \alpha_{11} (a_{i,j},\theta_{i,j})
(\delta_{xx} u)_{i,j}
- 2 \alpha_{12} (a_{i,j},\theta_{i,j})
(\delta_{xy}^{[disc]} u)_{i,j}
\right.
&
\\
\left.
- \alpha_{22} (a_{i,j},\theta_{i,j})
(\delta_{yy} u)_{i,j}
+ 2\sqrt{a_{i,j}(1-a_{i,j})f_{i,j}}
\right\}
&
= 0,
\end{array}
\end{equation}
where $disc=$ 1 or 2 if (\ref{eq:narrow_cond1}) or (\ref{eq:narrow_cond2}) is satisfied respectively.

\underline{Semi-Lagrangian wide stencil discretization.} Otherwise, the discrete equation is given by
\begin{equation}
\label{eq:disc_wide2}
\renewcommand*{\arraystretch}{1}
\begin{array}{rl}
\displaystyle\max_{(a_{i,j},\theta_{i,j})\in \Gamma}
\left\{
\frac{•}{•}
- a_{i,j} \,
(\delta_{zz} u)_{i,j}
- \left( 1-a_{i,j} \right) \,
(\delta_{ww} u)_{i,j}
\right.
\hspace{1cm}
&
\\
\left.
+ 2\sqrt{
a_{i,j} \left( 1-a_{i,j} \right) f_{i,j}
}
\right\}
&
= 0,
\end{array}
\end{equation}
where $(\delta_{zz} u)_{i,j}$ and $(\delta_{ww} u)_{i,j}$ are defined by (\ref{eq:derivative_wide1}) and (\ref{eq:derivative_wide2}) when $\mathbf{x}_{i,j}$ is inside the computational domain, and by (\ref{eq:derivative_wide1_bc1}) or similar expressions when $\mathbf{x}_{i,j}$ is near the boundary.

\subsection{The nonlinear discrete system}
\label{subsec:discrete}

The mixed discretization scheme, defined by (\ref{eq:disc_narrow}) and (\ref{eq:disc_wide2}), gives rise to a nonlinear discrete system that contains $N^2$ discrete equations. If we define a vector of the unknowns
$
u_h \equiv
(
u_{1,1},
u_{1,2},
\cdots,
u_{1,N},
u_{2,1},
\cdots
\cdots,
u_{N,N}
)^T
\in \mathbb{R}^{N^2\times 1}
$,
and similarly, vectors of controls $a_h\in \mathbb{R}^{N^2\times 1}$, $\theta_h\in \mathbb{R}^{N^2\times 1}$, then the entire nonlinear discrete system can be written into the following matrix form:
\begin{equation}
\label{eq:discretecomplete}
\displaystyle\max_{(a_h,\theta_h)\in\Gamma}
\left\{
\mathbf{A}(a_h,\theta_h) \, u_h
- F_h(a_h,\theta_h)
\right\}
= 0,
\end{equation}
where $\mathbf{A}(a_h,\theta_h)\in \mathbb{R}^{N^2\times N^2}$ is a matrix that consists of the coefficients of $u^h$, and $F_h(a_h,\theta_h)\in \mathbb{R}^{N^2\times 1}$ is a vector that does not explicitly contain $u^h$. We note that this nonlinear system can be treated as a combination of an optimization problem and a linear system as follows:
\begin{equation}
\label{eq:discretecomplete1}
\mathcal{F}_h (u_h) \equiv
\displaystyle\max_{(a_h,\theta_h)\in\Gamma} \mathcal{L}_h (a_h,\theta_h; u_h) = 0,
\end{equation}
where the to-be-maximized linear system is
\begin{equation}
\label{eq:discretecomplete2}
\mathcal{L}_h (a_h,\theta_h; u_h) \equiv
\mathbf{A}(a_h,\theta_h) \, u_h
- F_h(a_h,\theta_h).
\end{equation}
Here the symbols $\mathcal{F}_h$ and $\mathcal{L}_h$ in (\ref{eq:discretecomplete1})-(\ref{eq:discretecomplete2}) represent the discretization of $\mathcal{F}$ and $\mathcal{L}$ in (\ref{eq:HJB2})-(\ref{eq:HJBdiff}), respectively.

%\subsection{The nonlinear discrete system: four cases}
%\label{subsec:discrete4}

\begin{figure}[t!]
\begin{center}
\includegraphics[scale=0.22]{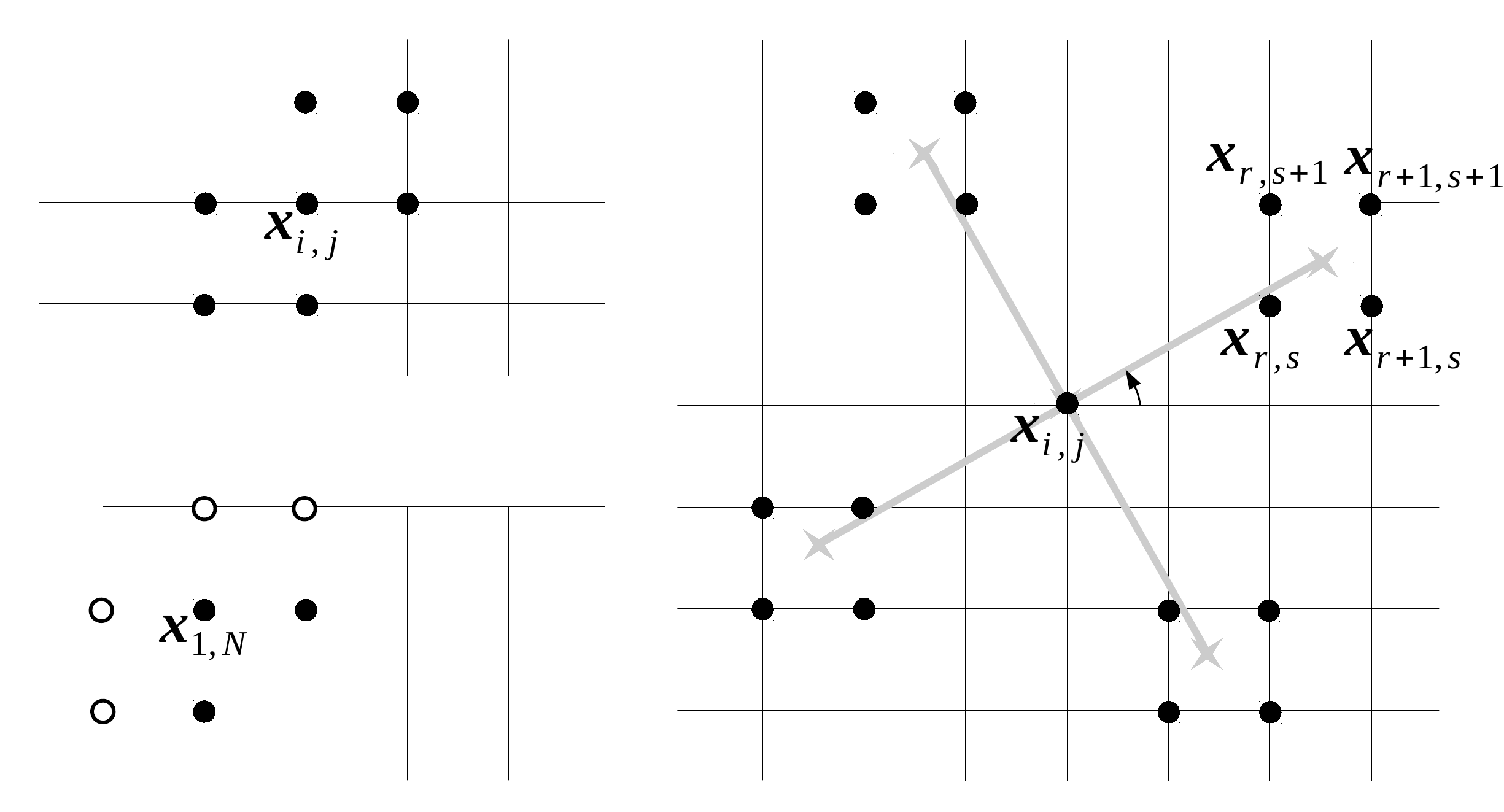}
\end{center}
\caption{\label{fig:linear_stencil}
(left-top) Case 1: Suppose Condition (\ref{eq:narrow_cond1}) is satisfied at $\mathbf{x}_{i,j}$ and the standard 7-point stencil discretization (\ref{eq:disc_narrow}) is used. The discrete equation contains 7 unknown values of $u_h$, labelled by the black dots. (left-bottom) Case 2: Consider $\mathbf{x}_{1,N}$, which is close to the boundary. The hollow dots sit on the boundary and the values of $u$ on these points are determined by the Dirichlet boundary condition. As a result, the discrete equation contains 3 unknown values of $u_h$, labeled by the black dots. (right) Case 3: Suppose neither (\ref{eq:narrow_cond1}) nor (\ref{eq:narrow_cond2}) is satisfied at $\mathbf{x}_{i,j}$ and thus semi-Lagrangian wide stencil discretization (\ref{eq:disc_wide2}) is used. Since bilinear interpolation of each stencil point contains 4 unknown values, the resulting discrete equation has 17 unknown values in total (black dots).}
\end{figure}

To show how the standard 7-point stencil discretization (\ref{eq:disc_narrow}) and the semi-Lagrangian wide stencil discretization (\ref{eq:disc_wide2}) can be written into the general form (\ref{eq:discretecomplete}), we analyze four cases.

\underline{Standard 7-point stencil discretization, grid point $\mathbf{x}_{i,j}$ inside $\Omega$.}
\hfill
Suppose Condition (\ref{eq:narrow_cond1}) is satisfied at $\mathbf{x}_{i,j}$. Then we use the standard 7-point stencil discretization (\ref{eq:disc_narrow}) with $disc=1$. This is illustrated in Figure \ref{fig:linear_stencil} (left-top). Some simple algebra shows that (\ref{eq:disc_narrow}) can be transformed into (\ref{eq:discretecomplete}) where
\begin{equation}
\label{eq:discretegridpoints_ex1}
\renewcommand*{\arraystretch}{1.8}
\begin{array}{l}
(\mathbf{A} u_h)_{i,j}
= \dfrac{2}{h^2}(\alpha_{11}+\alpha_{22}-\alpha_{12})u_{i,j}
-\dfrac{1}{h^2}(\alpha_{11}-\alpha_{12})u_{i+1,j}
\\
\quad
-\dfrac{1}{h^2}(\alpha_{11}-\alpha_{12})u_{i-1,j}
-\dfrac{1}{h^2}(\alpha_{22}-\alpha_{12})u_{i,j+1}
\\
\quad
-\dfrac{1}{h^2}(\alpha_{22}-\alpha_{12})u_{i,j-1}
-\dfrac{1}{h^2}\alpha_{12} \, u_{i+1,j+1}
-\dfrac{1}{h^2}\alpha_{12} \, u_{i-1,j-1},
\\
F_{i,j}
= -2\sqrt{a_{i,j}(1-a_{i,j})f_{i,j}},
\end{array}
\end{equation}
where $(\mathbf{A} u_h)_{i,j}$ and $F_{i,j}$ are the values of $\mathbf{A}(a_h,\theta_h) u_h$ and $F_h(a_h,\theta_h)$ at the grid point $\mathbf{x}_{i,j}$. For simplicity, we have suppressed the dependency of $\mathbf{A}$, $F_{i,j}$, $\alpha_{11}$, $\alpha_{22}$ and $\alpha_{12}$ on $(a_{i,j},\theta_{i,j})$. This equation contains 7 unknown values of $u_h$. Similarly, interested readers can also write down the expressions when Condition (\ref{eq:narrow_cond2}) is satisfied at $\mathbf{x}_{i,j}$ and the standard 7-point stencil discretization (\ref{eq:disc_narrow}) with $disc=2$ is applied.

\underline{Standard 7-point stencil discretization, grid point $\mathbf{x}_{i,j}$ near $\partial\Omega$.}
\hfill
Without loss of generality, we assume that $\mathbf{x}_{i,j}=\mathbf{x}_{1,N}$, as shown in Figure \ref{fig:linear_stencil} (left-bottom). Now $u_{i-1,j}$, $u_{i,j+1}$, $u_{i+1,j+1}$ and $u_{i-1,j-1}$ can be determined by the Dirichlet boundary condition $u=g$. These terms become part of $F_{i,j}$. As a result, $(\mathbf{A} u_h)_{i,j}$ contains only 3 unknown values.

%\underline{Standard 7-point stencil discretization, grid point $\mathbf{x}_{i,j}$ near $\partial\Omega$.}
%\hfill
%Without loss of generality, we assume that $\mathbf{x}_{i,j}=\mathbf{x}_{1,N}$, as shown in Figure \ref{fig:linear_stencil} (left-bottom). Now $u_{i-1,j}$, $u_{i,j+1}$, $u_{i+1,j+1}$ and $u_{i-1,j-1}$ can be determined by the Dirichlet boundary condition $u=g$ in (\ref{eq:MAE2}). These terms become part of $F_{i,j}$.
%Then
%\begin{equation}
%\label{eq:discretegridpoints_ex2}
%\renewcommand*{\arraystretch}{1.8}
%\begin{array}{l}
%(\mathbf{A} u_h)_{i,j}
%= \dfrac{2}{h^2}(\alpha_{11}+\alpha_{22}-\alpha_{12})u_{i,j}
%-\dfrac{1}{h^2}(\alpha_{11}-\alpha_{12})u_{i+1,j}
%-\dfrac{1}{h^2}(\alpha_{22}-\alpha_{12})u_{i,j-1},
%\\
%F_{i,j}
%= -2\sqrt{a_{i,j}(1-a_{i,j})f_{i,j}}
%+\dfrac{1}{h^2}(\alpha_{11}-\alpha_{12})g_{i-1,j}
%+\dfrac{1}{h^2}(\alpha_{22}-\alpha_{12})g_{i,j+1}
%\\
%\hspace{7cm}
%+\dfrac{1}{h^2}\alpha_{12} \, g_{i+1,j+1}
%+\dfrac{1}{h^2}\alpha_{12} \, g_{i-1,j-1}.
%\end{array}
%\end{equation}
%As a result, (\ref{eq:discretegridpoints_ex2}) contains only 3 unknown values.

\underline{Semi-Lagrangian wide stencil discretization, grid point $\mathbf{x}_{i,j}$ inside $\Omega$.}
\hfill
Suppose neither (\ref{eq:narrow_cond1}) nor (\ref{eq:narrow_cond2}) is fulfilled at $\mathbf{x}_{i,j}$, so semi-Lagrangian wide stencil discretization (\ref{eq:disc_wide2}) is applied; see Figure \ref{fig:linear_stencil} (right). Then (\ref{eq:disc_wide2}) can be written into (\ref{eq:discretecomplete}) where
\begin{equation}
\label{eq:discretegridpoints_ex3}
\renewcommand*{\arraystretch}{1.8}
\begin{array}{l}
(\mathbf{A} u_h)_{i,j}
= \dfrac{2}{h} u_{i,j}
-\dfrac{a_{i,j}}{h}
\left.
\mathcal{I}_h u
\right|_{\mathbf{x}_{i,j}+\sqrt{h}(\mathbf{e}_z)_{i,j}}
-\dfrac{a_{i,j}}{h}
\left.
\mathcal{I}_h u
\right|_{\mathbf{x}_{i,j}-\sqrt{h}(\mathbf{e}_z)_{i,j}}
\\
\hspace{2cm}
-\dfrac{1-a_{i,j}}{h}
\left.
\mathcal{I}_h u
\right|_{\mathbf{x}_{i,j}+\sqrt{h}(\mathbf{e}_w)_{i,j}}
-\dfrac{1-a_{i,j}}{h}
\left.
\mathcal{I}_h u
\right|_{\mathbf{x}_{i,j}-\sqrt{h}(\mathbf{e}_w)_{i,j}},
\\
F_{i,j}
= -2\sqrt{a_{i,j}(1-a_{i,j})f_{i,j}}.
\end{array}
\end{equation}
We note that each bilinear interpolation term contains 4 unknowns. For instance,
$
\left.
\mathcal{I}_h u
\right|_{\mathbf{x}_{i,j}+\sqrt{h}(\mathbf{e}_z)_{i,j}}
$
can be written as the linear combination of the unknowns at the four neighboring points $u_{r,s}$, $u_{r+1,s}$, $u_{r,s+1}$ and $u_{r+1,s+1}$, which are labeled in Figure \ref{fig:linear_stencil} (right). As a result, (\ref{eq:discretegridpoints_ex3}) has 17 unknown values.

\underline{Semi-Lagrangian wide stencil discretization, grid point $\mathbf{x}_{i,j}$ near $\partial\Omega$.}
The analysis is similar to the previous cases. The number of the unknowns is less than 17.

%\underline{Semi-Lagrangian wide stencil discretization, grid point $\mathbf{x}_{i,j}$ near $\partial\Omega$.}
%The analysis is similar to the previous cases. Once again, assume that the wide stencil points $\mathbf{x}_{i,j}+\sqrt{h}(\mathbf{e}_z)_{i,j}$ falls outside $\overline{\Omega}$ and is relocated to $\mathbf{x}_{i,j}+\eta_1(\mathbf{e}_z)_{i,j}\in\partial\Omega$. Then
%\begin{equation}
%\label{eq:discretegridpoints_ex4}
%\renewcommand*{\arraystretch}{1.8}
%\begin{array}{l}
%(\mathbf{A} u_h)_{i,j}
%= 2\left(
%\dfrac{a_{i,j}}{\eta_1 \sqrt{h}} + \dfrac{1-a_{i,j}}{h}
%\right)
%u_{i,j}
%-\dfrac{a_{i,j}}{\sqrt{h} \frac{\eta_1+\sqrt{h}}{2}}
%\left.
%\mathcal{I}_h u
%\right|_{\mathbf{x}_{i,j}-\sqrt{h}(\mathbf{e}_z)_{i,j}}
%\\
%\hspace{2cm}
%-\dfrac{1-a_{i,j}}{h}
%\left.
%\mathcal{I}_h u
%\right|_{\mathbf{x}_{i,j}+\sqrt{h}(\mathbf{e}_w)_{i,j}}
%-\dfrac{1-a_{i,j}}{h}
%\left.
%\mathcal{I}_h u
%\right|_{\mathbf{x}_{i,j}-\sqrt{h}(\mathbf{e}_w)_{i,j}},
%\\
%F_{i,j}
%= -2\sqrt{a_{i,j}(1-a_{i,j})f_{i,j}}
%+\dfrac{a_{i,j}}{\eta_1 \frac{\eta_1+\sqrt{h}}{2}}
%g(\mathbf{x}_{i,j}+\eta_1 (\mathbf{e}_z)_{i,j}).
%\end{array}
%\end{equation}
%We note that the number of the unknowns is less than 17.

%% Solving the Nonlinear Discrete System
%% ==============

\section{Solving the Nonlinear Discrete System}
\label{sec:policy}

\subsection{Policy iteration}

After setting up the complete nonlinear discrete system (\ref{eq:discretecomplete1})-(\ref{eq:discretecomplete2}), the next objective is to solve it. We apply a well-known fixed point iteration algorithm, called policy iteration (or Howard's algorithm) \cite{howard1960dynamic,forsyth2007numerical} as follows:

\begin{enumerate}[leftmargin=8mm]

\item Start with an initial guess of the solution $u_h^{(0)}$.

\item For $k=0,1,...$ until convergence:

	\begin{enumerate}[leftmargin=5mm]
	
	\item Solve for the optimal control pair
	$( a_h^{(k)},\theta_h^{(k)} )$
	under the current solution $u_h^{(k)}$:
	\begin{equation}
	\label{eq:policy_step1}
	( a_{i,j}^{(k)},\theta_{i,j}^{(k)} )
	=
	\displaystyle
	\argmax_{\substack{(a_{i,j},\theta_{i,j}) \in\Gamma_{i,j}}}
	\mathcal{L}_{i,j} ( a_{i,j},\theta_{i,j}; u_h^{(k)} ),
	\quad
	\text{for all } \mathbf{x}_{i,j} \in \Omega,
	\end{equation}
	where $\mathcal{L}_{i,j}$ is the pointwise component of $\mathcal{L}_h\in\mathbb{R}^{N^2\times 1}$ defined in (\ref{eq:discretecomplete2}) and $\Gamma_{i,j}=[0,1]\times[-\frac{\pi}{4},\frac{\pi}{4})$ is the control set at $\mathbf{x}_{i,j}$.
	
	Meanwhile, obtain the residual $R_h^{(k)}\in\mathbb{R}^{N^2\times 1}$, where each pointwise component reads
	$
	R_{i,j}^{(k)} \equiv \mathcal{L}_{i,j} ( a_{i,j}^{(k)},\theta_{i,j}^{(k)}; u_h^{(k)} )
	$.
	
	\item If $\|R_h^{(k)}\| \leq \text{tolerance}$: break

	Else, solve the following linear system for the solution $u_h^{(k+1)}$ under the current optimal control pair $(a_h^{(k)},\theta_h^{(k)})$:
	\begin{equation}
	\label{eq:policy_step2}
	\mathbf{A}(a_h^{(k)},\theta_h^{(k)}) \, u_h^{(k+1)}
	= F_h(a_h^{(k)},\theta_h^{(k)})
	\\
	\qquad
	\Rightarrow
	\qquad
	u_h^{(k+1)}.
	\end{equation}
	
	\end{enumerate}
	
%\item Final solution:
%$
%u_h = u_h^{(k)}
%$,
%$
%(a_h^*,\theta_h^*) = (a_h^{(k)},\theta_h^{(k)})
%$.
\end{enumerate}

It is proved that policy iteration is guaranteed to converge for any initial guess $u_h^{(0)}$, if by applying a monotone discretization to an HJB equation, the resulting matrix $\mathbf{A}(a_h,\theta_h)$ is an M-matrix under all admissible controls \cite{bokanowski2009some,azimzadeh2016weakly}. We will show in Section \ref{subsec:stability} that the resulting matrix $\mathbf{A}(a_h,\theta_h)$ in (\ref{eq:discretecomplete}) is indeed an M-matrix.

Policy iteration consists of two sub-steps. One sub-step is to solve the linear system under a given control pair; see (\ref{eq:policy_step2}). We use Krylov subspace methods, such as the GMRES with the incomplete LU preconditioner.
%To give a rough estimation of the computational cost, we assume that the matrix $\mathbf{A}(a_h,\theta_h)$ resembles a (rotated) discrete 2D Laplacian operator. Also, we consider that in practice, the standard 7-point discretization can be applied on either the entire $\Omega$ or most of $\Omega$ (see the discussion at the end of Section \ref{subsec:mixed}). Then the condition number is $\kappa = O(h^{-2})$ and the computational cost is $O(\#\Omega \sqrt{\kappa}) = O(\#\Omega^{3/2})$, where $\#\Omega=N^2$ denotes the total number of grid points.
The other sub-step of the policy iteration is to solve the optimization problem at each grid point $\mathbf{x}_{i,j}$; see (\ref{eq:policy_step1}).
%It turns out that computing the optimal controls dominates the overall computational cost. 
We will discuss speeding up computation of the optimization problem in detail in the next section.

\subsection{Speeding up computation of optimal controls}
\label{subsec:controls}

\begin{figure}[t!]
\begin{center}
\includegraphics[scale=0.5]{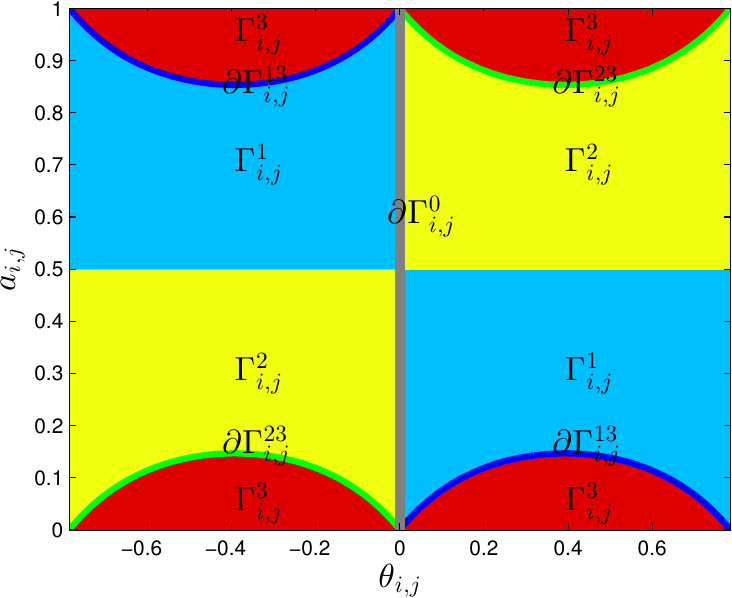}
\end{center}

\begin{center}
\small
\setlength\tabcolsep{2pt}
\begin{tabular}{|c|c|c|c|c|c|}
\hline
Region
&
Definition
&
Discretization
&
\begin{tabular}{c}
Optimization
\\
algorithm
\\
in each region
\end{tabular}
&
Cost
&
\begin{tabular}{c}
Extra
\\
truncation
\\
error
\\
introduced?
\end{tabular}
\\
\hline
$\Gamma_{i,j}^1$
&
\begin{tabular}{c}
The region where
\\
Condition (\ref{eq:narrow_cond1})
\\
is satisfied
\end{tabular}
&
\begin{tabular}{c}
Standard
\\
7-point stencil
\\
with $disc=1$
\end{tabular}
&
\multirow{2}{*}{
\begin{tabular}{c}
Closed-form 
\\
formula from
\\
first derivative test
\end{tabular}
}
&
\multirow{2}{*}{
$O(1)$
}
&
\multirow{2}{*}{
No
}
\\
\cline{1-3}
$\Gamma_{i,j}^2$
&
\begin{tabular}{c}
The region where
\\
Condition (\ref{eq:narrow_cond2})
\\
is satisfied
\end{tabular}
&
\begin{tabular}{c}
Standard
\\
7-point stencil
\\
with $disc=2$
\end{tabular}
&
&
&
\\
\hline
$\Gamma_{i,j}^{3}$
&
\begin{tabular}{c}
The region where
\\
neither (\ref{eq:narrow_cond1}) nor
\\
(\ref{eq:narrow_cond2}) is satisfied
\end{tabular}
&
\begin{tabular}{c}
Semi-Lagrangian
\\
wide stencil
\end{tabular}
&
\begin{tabular}{c}
Linear search over
\\
a single control
\\
$\theta_{i,j}\in[-\frac{\pi}{4},\frac{\pi}{4})$
\end{tabular}
&
$O(M)$
&
Yes
\\
\hline
$\partial\Gamma_{i,j}^0$
&
The line $\theta_{i,j}=0$
&
\begin{tabular}{c}
Standard
\\
7-point stencil
\\
with $disc=1$ or $2$
\end{tabular}
&
\multirow{3}{*}{
\begin{tabular}{c}
Closed-form 
\\
formula from
\\
first derivative test
\end{tabular}
}
&
\multirow{3}{*}{
$O(1)$
}
&
\multirow{3}{*}{
No
}
\\
\cline{1-3}
$\partial\Gamma_{i,j}^{13}$
&
\begin{tabular}{c}
The boundary
\\
between $\Gamma_{i,j}^1$
\\
and $\Gamma_{i,j}^3$
\end{tabular}
&
\begin{tabular}{c}
Standard
\\
7-point stencil
\\
with $disc=1$
\end{tabular}
&
&
&
\\
\cline{1-3}
$\partial\Gamma_{i,j}^{23}$
&
\begin{tabular}{c}
The boundary
\\
between $\Gamma_{i,j}^2$
\\
and $\Gamma_{i,j}^3$
\end{tabular}
&
\begin{tabular}{c}
Standard
\\
7-point stencil
\\
with $disc=2$
\end{tabular}
&
&
&
\\
\hline
\end{tabular}
\end{center}
\caption{\label{fig:control}
Division of the admissible control set $\Gamma_{i,j}=[0,1]\times[-\frac{\pi}{4},\frac{\pi}{4})$ into regions. For each region, the characterization, discretization, optimization algorithm and the corresponding cost / truncation error of the optimization algorithm are listed.}
\end{figure}

Since the semi-Lagrangian wide stencil discretization of $(\delta_{zz} u)_{i,j}$ and $(\delta_{ww} u)_{i,j}$ in (\ref{eq:disc_wide2}) depends on the control $\theta_{i,j}$, there is no simple closed-form formula to evaluate the optimal $(a_{i,j}^{(k)},\theta_{i,j}^{(k)})$ directly. In this case, one typical approach is to use bilinear search algorithm for the optimization problem. More specifically, consider the optimization problem at a grid point $\mathbf{x}_{i,j}$. We discretize the continuous admissible control set $\Gamma_{i,j}=[0,1]\times [-\frac{\pi}{4},\frac{\pi}{4})$ into an $M\times M$ discrete set, denoted as $\Gamma_{i,j}^h$. We note that the discretization of the control set introduces additional truncation error. In order to maintain consistency, we must let $M\to\infty$ as $h\to 0$. A typical choice of $M$ is $M=N$. Then we compute the $M\times M$ values of the objective function $\mathcal{L}_{i,j}(a_{i,j},\theta_{i,j};u_h^{(k)})$ with $(a_{i,j},\theta_{i,j})\in\Gamma_{i,j}^h$ and then find the global maximal value, which gives the optimal $(a_{i,j}^{(k)},\theta_{i,j}^{(k)})$. However, the computational cost of the bilinear search per grid point $\mathbf{x}_{i,j}$ is $O(M^2)$. Furthermore, if we denote the total number of grid points as $\#\Omega=N^2$, then the computational cost on the entire computational domain $\Omega$ is as high as $O(M^2\#\Omega)$, or $O(\#\Omega^2)$ if we choose $M=N$.
%As a consequence, in one policy iteration, solving the optimization problem dominates solving the linear system in terms of computational cost.

In order to speed up computation for the optimal controls, we divide the continuous admissible control set $\Gamma_{i,j}=[0,1]\times [-\frac{\pi}{4},\frac{\pi}{4})$ into six regions, as shown in Figure \ref{fig:control}\footnote{It is unnecessary to consider the line $a_{i,j}=\frac{1}{2}$, since the objective function is a constant on this line. Also it is unnecessary to consider the line $\theta_{i,j} = \pm\frac{\pi}{4}$, since $\mathcal{L}_{a,\theta} \, u = \mathcal{L}_{1-a,\theta+\frac{\pi}{2}} \, u$ indicates that $\theta_{i,j} = \pm\frac{\pi}{4}$ is indeed an interior part of $\Gamma_{i,j}^1$ and $\Gamma_{i,j}^2$.}.
The six regions are identified by whether a control pair $(a_{i,j},\theta_{i,j})$ satisfies (\ref{eq:narrow_cond1}), or (\ref{eq:narrow_cond2}), or neither.
Our approach is to find the optimal control pair within each region, and then find the global optimal control pair among the six regional optimal control pairs. This approach enables us to make full use of the analytical property of each region, and to improve the optimization algorithm within each region and eventually on the entire admissible control set $\Gamma_{i,j}$.

Using our approach, the computational cost of solving the optimization problem on $\Gamma_{i,j}$ can be significantly reduced. More specifically, if the standard 7-point stencil discretization can be applied monotonically on all or most of the grid points, then the computational cost is $O(1)$ per grid point and $O(\#\Omega)$ on the entire computational domain. In general, the computational cost is at most $O(M)$ per grid point and at most $O(M\#\Omega)$ on the entire computational domain. For the typical choice $M=N$, the total computational cost of solving the optimization problem is $O(\#\Omega^{3/2})$.
%, which is comparable to the step of solving the linear system.

To explain the details of the regional optimization, consider again a given grid point $\mathbf{x}_{i,j}$ and its associated control set $\Gamma_{i,j}$. In Region $\Gamma_{i,j}^1$, $\Gamma_{i,j}^2$, $\partial\Gamma_{i,j}^0$, $\partial\Gamma_{i,j}^{13}$ and $\partial\Gamma_{i,j}^{23}$ (see Figure \ref{fig:control}), where the standard 7-point stencil discretization (\ref{eq:disc_narrow}) is applied, the discretization of $(\delta_{xx}u)_{i,j}$, $(\delta_{yy}u)_{i,j}$, $(\delta_{xy}^{[1]}u)_{i,j}$ and $(\delta_{xy}^{[2]}u)_{i,j}$ does not depend on the controls $(a_{i,j},\theta_{i,j})$. This enables us to derive a closed-form formula for the optimal controls in these regions using first derivative test, which can be evaluated by $O(1)$ operation and introduces no additional truncation error. More specifically:

\underline{Region $\Gamma_{i,j}^1$.} The region is defined where Condition (\ref{eq:narrow_cond1}) is satisfied. Equation (\ref{eq:disc_narrow}) gives the objective function in $\Gamma_{i,j}^1$:
\begin{equation}
\label{eq:regionanalysis1_1}
\renewcommand*{\arraystretch}{1.2}
\begin{array}{rl}
\mathcal{L}_{i,j} (a_{i,j},\theta_{i,j})
&
= -\alpha_{11} (a_{i,j},\theta_{i,j})
(\delta_{xx} u)_{i,j}
- 2 \alpha_{12} (a_{i,j},\theta_{i,j})
(\delta_{xy}^{[1]} u)_{i,j}
\\
&
- \alpha_{22} (a_{i,j},\theta_{i,j})
(\delta_{yy} u)_{i,j}
+ 2\sqrt{a_{i,j}(1-a_{i,j})f_{i,j}},
\end{array}
\end{equation}
where we only manifest the dependency of $\mathcal{L}_{i,j}$ on the control pair $(a_{i,j},\theta_{i,j})$.
One can verify that this function is smooth in $(a_{i,j},\theta_{i,j})\in \Gamma_{i,j}^1$, concave in $a_{i,j}\in[0,1]$, and its stationary point in $\Gamma_{i,j}^1$ is unique, if it exists. This allows us to use first derivative test to find the optimal control pair in $\Gamma_{i,j}^1$:
\begin{equation}
\label{eq:regionanalysis1_2}
\theta_{i,j}^* = \frac{1}{2} \arctan
\frac {2(\delta_{xy}^{[1]} u)_{i,j}} {(\delta_{yy} u)_{i,j}-(\delta_{xx} u)_{i,j}},
\quad
a_{i,j}^* =
\frac{1}{2}\left(1-\frac{\lambda_{i,j}}{\sqrt{4f_{i,j}+\lambda_{i,j}^2}}\right),
\end{equation}
where
$
\lambda_{i,j} \equiv [(\delta_{xx} u)_{i,j}-(\delta_{yy} u)_{i,j}] \cos 2\theta_{i,j}^* 
-2(\delta_{xy}^{[1]} u)_{i,j} \sin 2\theta_{i,j}^*
$.
With a slight abuse of notations, here and for the rest of Section \ref{subsec:controls}, we use $(a_{i,j}^*,\theta_{i,j}^*)$ to denote the the regional (rather than global) optimal control pair at $\mathbf{x}_{i,j}$. We note that $(a_{i,j}^*,\theta_{i,j}^*)$ given by (\ref{eq:regionanalysis1_2}) may not necessarily be inside $\Gamma_{i,j}^1$. If $(a_{i,j}^*,\theta_{i,j}^*)\in\Gamma_{i,j}^1$, then the maximum in $\Gamma_{i,j}^1$ must occur at $(a_{i,j}^*,\theta_{i,j}^*)$. Otherwise, the maximum must occur on the boundary of $\Gamma_{i,j}^1$, or more specifically, either $\partial\Gamma_{i,j}^0$ or $\partial\Gamma_{i,j}^{13}$, which will be investigated separately.

\underline{Region $\Gamma_{i,j}^2$.} The region is defined where Condition (\ref{eq:narrow_cond2}) is satisfied. The analysis for solving the optimization problem in $\Gamma_{i,j}^2$ is the same as $\Gamma_{i,j}^1$, except that $(\delta_{xy}^{[1]} u)_{i,j}$ in (\ref{eq:regionanalysis1_1}), (\ref{eq:regionanalysis1_2}) is replaced by $(\delta_{xy}^{[2]} u)_{i,j}$.

\underline{Region $\partial\Gamma_{i,j}^0$.} This is the line $\theta_{i,j}=0$ which separates Region $\Gamma_{i,j}^1$ and $\Gamma_{i,j}^2$. The objective function in $\partial\Gamma_{i,j}^0$ can be found in (\ref{eq:regionanalysis1_1}), where $\alpha_{12}=0$ and thus the cross derivative term disappears. The optimal control pair in $\partial\Gamma_{i,j}^0$ is simply
\begin{equation}
\label{eq:regionanalysis0_2}
\theta_{i,j}^* = 0,
\quad
a_{i,j}^* =
\frac{1}{2}\left[
1-\frac{(\delta_{xx} u)_{i,j}-(\delta_{yy} u)_{i,j}}
{\sqrt{4f_{i,j}+((\delta_{xx} u)_{i,j}-(\delta_{yy} u)_{i,j})^2}}
\right].
\end{equation}

\underline{Region $\partial\Gamma_{i,j}^{13}$.} This is the boundary between Region $\Gamma_{i,j}^1$ and $\Gamma_{i,j}^3$.
If we define the signs of $a_{i,j}-\frac{1}{2}$ and $\theta_{i,j}$ as
\begin{equation}
s_{a-1/2} \equiv \left\{
\renewcommand*{\arraystretch}{1.2}
\begin{array}{ll}
-1, & a_{i,j}-\frac{1}{2} < 0,
\\
1, & a_{i,j}-\frac{1}{2} > 0,
\end{array}
\right.
\quad
s_\theta \equiv \left\{
\renewcommand*{\arraystretch}{1}
\begin{array}{ll}
-1, & \theta_{i,j} < 0,
\\
1, & \theta_{i,j} > 0,
\end{array}
\right.
\end{equation}
then $\partial\Gamma_{i,j}^{13}$ contains two sections: (i) $(s_{a-1/2},s_\theta)=(1,-1)$, (ii) $(s_{a-1/2},s_\theta)=(-1,1)$.

The objective function on $\partial\Gamma_{i,j}^{13}$ is the same as (\ref{eq:regionanalysis1_1}). First derivative test shows that for each of the two sections of $\partial\Gamma_{i,j}^{13}$, the maximum of the objective function occurs at
\begin{equation}
\theta_{i,j}^* =
\frac{s_\theta}{2} \arctan \left(
1 + \gamma_{i,j}^2 - \gamma_{i,j} \sqrt{2+\gamma_{i,j}^2}
\right),
\end{equation}
where
$
\gamma_{i,j} \equiv
\frac{s_{a-1/2}}{2\sqrt{f_{i,j}}}
\left( (\delta_{yy} u)_{i,j} - (\delta_{xx} u)_{i,j} - 2 s_\theta (\delta_{xy}^{[1]} u)_{i,j} \right)
$.
The corresponding $a_{i,j}^*\in\partial\Gamma_{i,j}^{13}$, derived from Condition (\ref{eq:narrow_cond1}), is
\begin{equation}
a_{i,j}^* =
\frac{1}{2} \left(
1 + \frac{s_{a-1/2}}{ \sqrt{2}\sin (2 |\theta_{i,j}^*|+\frac{\pi}{4} ) }
\right).
\end{equation}

\underline{Region $\partial\Gamma_{i,j}^{23}$.} This is the boundary between Region $\Gamma_{i,j}^2$ and $\Gamma_{i,j}^3$. The analysis on $\partial\Gamma_{i,j}^{23}$ is then the same as $\partial\Gamma_{i,j}^{13}$, except that the two sections of $\partial\Gamma_{i,j}^{23}$ become (i) $(s_{a-1/2},s_\theta)=(1,1)$, (ii) $(s_{a-1/2},s_\theta)=(-1,-1)$, and $(\delta_{xy}^{[1]} u)_{i,j}$ is replaced by $(\delta_{xy}^{[2]} u)_{i,j}$.

\underline{Region $\Gamma_{i,j}^3$.} The region is defined where neither (\ref{eq:narrow_cond1}) nor (\ref{eq:narrow_cond2}) is satisfied. The semi-Lagrangian wide stencil discretization (\ref{eq:disc_wide2}) is applied. Accordingly, the objective function reads
\begin{equation}
\label{eq:regionanalysis_3_1}
\mathcal{L}_{i,j} (a_{i,j},\theta_{i,j})
= -a_{i,j} \,
(\delta_{zz} u)_{i,j}
- ( 1-a_{i,j} ) \,
(\delta_{ww} u)_{i,j}
+ 2\sqrt{
a_{i,j} ( 1-a_{i,j} ) f_{i,j}
}.
\end{equation}
The dependency of the discretization of $(\delta_{zz} u)_{i,j}$ and $(\delta_{ww} u)_{i,j}$ on the control $\theta_{i,j}$ prevents us from deriving a closed-form formula for $\theta_{i,j}^*\in\Gamma_{i,j}^3$.
%However, if we choose a mixed discretization, then semi-Lagrangian wide stencil discretization is restricted in Region $\Gamma_{i,j}^3$, which reflects the fact that usually only certain points in the computational domain (such as singular points) requires the semi-Lagrangian wide stencil discretization.
However, we note that the discretization of $(\delta_{zz} u)_{i,j}$ and $(\delta_{ww} u)_{i,j}$ is independent of the control $a_{i,j}$, which implies that a two dimensional bilinear search on the controls $(a_{i,j},\theta_{i,j})\in\Gamma_{i,j}$ can be reduced to a one-dimensional linear search on the single control $\theta_{i,j}\in [-\frac{\pi}{4},\frac{\pi}{4})$.

One can prove that the regional optimal control pair $(a_{i,j}^*,\theta_{i,j}^*) \in \Gamma_{i,j}^3$ must sit on the following parametrized curve
\begin{equation}
\label{eq:regionanalysis_3_2}
a_{i,j}(\theta_{i,j}) =
\left\{
\renewcommand*{\arraystretch}{1.2}
\begin{array}{ll}
\mathcal{C}^{\lambda}(\theta_{i,j}),
&
\text{ if }
\mathcal{C}^{\lambda}(\theta_{i,j}) \leq \mathcal{C}^{-}(\theta_{i,j})
\text{ or }
\mathcal{C}^{\lambda}(\theta_{i,j}) \geq \mathcal{C}^{+}(\theta_{i,j}),
\\
\mathcal{C}^{-}(\theta_{i,j}),
&
\text{ if }
\mathcal{C}^{-}(\theta_{i,j}) \leq \mathcal{C}^{\lambda}(\theta_{i,j}) \leq \frac{1}{2},
\\
\mathcal{C}^{+}(\theta_{i,j}),
&
\text{ if }
\frac{1}{2} \leq \mathcal{C}^{\lambda}(\theta_{i,j}) \leq \mathcal{C}^{+}(\theta_{i,j}).
\end{array}
\right.
\end{equation}
Here the curves
\begin{equation}
\mathcal{C}^{\pm}(\theta_{i,j})
\equiv \frac{1}{2}\left( 1 \pm \frac{1}{ \sqrt{2}\sin (2 |\theta_{i,j}|+\frac{\pi}{4} ) } \right),
\quad
\theta_{i,j} \in [-\frac{\pi}{4},\frac{\pi}{4})
\end{equation}
are given by Condition (\ref{eq:narrow_cond1}) and (\ref{eq:narrow_cond2}). The other curve
\begin{equation}
\mathcal{C}^{\lambda}(\theta_{i,j})
\equiv \frac{1}{2}\left[
1-\frac{(\delta_{zz} u)_{i,j} - (\delta_{ww} u)_{i,j}}
{\sqrt{4f_{i,j}+((\delta_{zz} u)_{i,j} - (\delta_{ww} u)_{i,j})^2}}\right],
\quad
\theta_{i,j} \in [-\frac{\pi}{4},\frac{\pi}{4}),
\end{equation}
where the directions of $z$ and $w$ depend on $\theta_{i,j}$, is given by the first derivative test of (\ref{eq:regionanalysis_3_1}) with respect to $a_{i,j}$.
Taking the parametrization (\ref{eq:regionanalysis_3_2}) into account, the objective function (\ref{eq:regionanalysis_3_1}) becomes $\mathcal{L}_{i,j}(a_{i,j}(\theta_{i,j}),\theta_{i,j})$, which is a function of the single control variable $\theta_{i,j}\in [-\frac{\pi}{4},\frac{\pi}{4})$.
This motivates us to discretize the set
$
[-\frac{\pi}{4},\frac{\pi}{4})
$
into an $M$-element control set,
and perform a linear search for the maximum of the parametrized objective function
$
\mathcal{L}_{i,j}(a_{i,j}(\theta_{i,j}),\theta_{i,j})
$
over the single control variable
$
\theta_{i,j}\in[-\frac{\pi}{4},\frac{\pi}{4})
$.
The computational cost is thus reduced to $O(M)$.

Once we obtain the six regional optimal control pairs and their corresponding objective function values, we search within them for the global optimal control pair on $\Gamma_{i,j}$. This step is cheap and straightforward.

As a side remark, in Section 8 of \cite{feng2016convergent}, the authors discretize $\theta$ with 64 different angles, regardless of the mesh size $N$. Indeed, if $\theta$ is discretized with fixed number of angles, then the numerical scheme in \cite{feng2016convergent} is no longer consistent in theory. This is different from our scheme, where $\theta$ is discretized with $M$ angles, and we choose $M = N$ such that consistency is still maintained.

%% Convergence Analysis
%% ==============

\section{Convergence Analysis}
\label{sec:converge}

As proved by Barles and Souganidis \cite{barles1991convergence}, there are four sufficient conditions for the numerical scheme of a nonlinear PDE to converge in the viscosity sense. In this section, we will prove that our numerical scheme does fulfill all the four requirements and is therefore guaranteed to converge to the viscosity solution of (\ref{eq:MAE2}).

\subsection{Consistency}
\label{subsec:consistency}

One sufficient condition for convergence is consistency. Intuitively, consistency claims that the discretized equation of a PDE should be close to the continuous PDE. In particular, when $h\to 0$, the discretized equation should converge to the PDE. The main result of this subsection is to prove that our numerical scheme is consistent in the viscosity sense:

\begin{lemma}[Consistency]
\label{lm:consistency}
For the Monge-Amp\`ere equation
$\mathcal{F}
\left(
\mathbf{x}, u(\mathbf{x}), D^2 u(\mathbf{x})
\right) = 0$,
the numerical scheme
$\mathcal{F}_h
\left(
\mathbf{x}_{i,j},
u_h
\right) = 0$,
given in (\ref{eq:discretecomplete1})-(\ref{eq:discretecomplete2}), is consistent in the viscosity sense.
More specifically, for any function $\varphi(\mathbf{x})\in C^\infty(\overline{\Omega})$ with $\varphi_{i,j} \equiv \varphi(\mathbf{x}_{i,j})$ and
$\varphi_h \equiv
(
\varphi_{1,1},
\varphi_{1,2},
\cdots,
\varphi_{N,N}
)^T
\in \mathbb{R}^{N^2\times 1}$,
for any $\mathbf{\hat{x}}\in\overline{\Omega}$, and for $h$ and $\xi$ that are arbitrary small constants independent of $\mathbf{x}$,
we have
\begin{align}
\renewcommand*{\arraystretch}{1}
\displaystyle\limsup
_{\substack{h\to 0, \, \xi\to 0 \\ \mathbf{x}_{i,j}\to\mathbf{\hat{x}}}}
\mathcal{F}_h
(\mathbf{x}_{i,j}, \varphi_h + \xi)
\leq
\mathcal{F}^*
(\mathbf{\hat{x}}, \varphi(\mathbf{\hat{x}}), D^2\varphi(\mathbf{\hat{x}})),
\\
\displaystyle\liminf
_{\substack{h\to 0, \, \xi\to 0 \\ \mathbf{x}_{i,j}\to\mathbf{\hat{x}}}}
\mathcal{F}_h
(\mathbf{x}_{i,j}, \varphi_h + \xi)
\geq
\mathcal{F}_*
(\mathbf{\hat{x}}, \varphi(\mathbf{\hat{x}}), D^2\varphi(\mathbf{\hat{x}})).
\end{align}
\end{lemma}

In practise, we prove a sufficient condition for consistency, called local consistency, as follows:

\begin{lemma}[Local consistency]
\label{lm:localconsistency}
Under the assumptions in Lemma \ref{lm:consistency}, we have
\begin{equation}
\label{eq:localconsistency}
\renewcommand*{\arraystretch}{1.3}
\begin{array}{l}
\mathcal{F}
(\mathbf{x}_{i,j}, \varphi(\mathbf{x}_{i,j}), D^2 \varphi(\mathbf{x}_{i,j}))
- \mathcal{F}_h
(\mathbf{x}_{i,j}, \varphi_h + \xi)
\\
= \left\{
\renewcommand*{\arraystretch}{1}
\begin{array}{ll}
O(h^2) + O(\xi),
&
\text{standard 7-point stencil},
\\
O(h) + O(\xi),
&
\text{semi-Lagrangian wide stencil, with all the 4}
\\
&
\text{wide stencil points } \in\Omega,
\\
O(\sqrt{h}) + O(\xi),
&
\text{semi-Lagrangian wide stencil, otherwise}.
\end{array}
\right.
\end{array}
\end{equation}
\end{lemma}

\begin{proof}
We note that the proof with $\xi=0$ is equivalent to the proof with a general $\xi$. Such equivalence can be easily verified if we substitute $\varphi$ by $\varphi+\xi$ in the following proof. Hence, we will only prove the case where $\xi=0$.

\underline{Truncation error of the standard 7-point stencil discretization.}
Suppose the standard 7-point stencil discretization is applied at $\mathbf{x}_{i,j}$. It is easy to show that the truncation errors for $(\delta_{xx}\varphi)_{i,j}$, $(\delta_{yy}\varphi)_{i,j}$, $(\delta_{xy}^{[1]}\varphi)_{i,j}$ and $(\delta_{xy}^{[2]}\varphi)_{i,j}$ are all $O(h^2)$. Hence, the local truncation error of the discrete linear equation (\ref{eq:discretecomplete2}) is then
$
\mathcal{L}_{a(\mathbf{x}_{i,j}),\theta(\mathbf{x}_{i,j})}
\varphi(\mathbf{x}_{i,j})
- \mathcal{L}_h
(\mathbf{x}_{i,j}; a_{i,j}, \theta_{i,j}; \varphi_h)
= O(h^2)
$.
Furthermore, the local truncation error of the finite difference scheme at $\mathbf{x}_{i,j}$ is
\begin{equation}
\label{eq:localconsistency_narrow}
\renewcommand*{\arraystretch}{1}
\begin{array}{l}
\left|
\;
\mathcal{F}
(\mathbf{x}_{i,j}, \varphi(\mathbf{x}_{i,j}), D^2 \varphi(\mathbf{x}_{i,j}))
- \mathcal{F}_h
(\mathbf{x}_{i,j}, \varphi_h)
\;
\right|
\\
=
\left|
\;
\displaystyle\max_{
(a(\mathbf{x}_{i,j}),\theta(\mathbf{x}_{i,j}))
\in\Gamma
}
\mathcal{L}_{a(\mathbf{x}_{i,j}),\theta(\mathbf{x}_{i,j})}
\varphi(\mathbf{x}_{i,j})
-
\displaystyle\max_{
(a_{i,j},\theta_{i,j})
\in\Gamma
}
\mathcal{L}_h
(\mathbf{x}_{i,j}; a_{i,j}, \theta_{i,j}; \varphi_h)
\right|
\\
\leq
\;
\displaystyle\max_{
(a_{i,j},\theta_{i,j})
\in\Gamma
}
\left|
\,
\mathcal{L}_{a_{i,j},\theta_{i,j}}
\varphi(\mathbf{x}_{i,j})
- \mathcal{L}_h
(\mathbf{x}_{i,j}; a_{i,j}, \theta_{i,j}; \varphi_h)
\,
\right|
= O(h^2).
\end{array}
\end{equation}
The inequality comes from
$
\left|
\displaystyle\max_{x} f(x) - \displaystyle\max_{x} g(x)
\right|
\leq
\displaystyle\max_{x} | f(x) - g(x) |
$.

\underline{Truncation error of semi-Lagrangian wide stencil discretization.}
\hfill
Suppose semi-Lagrangian wide stencil discretization is applied at $\mathbf{x}_{i,j}$. We focus on the truncation error for $(\delta_{zz}\varphi)_{i,j}$ only and analyze three cases. The first case is that both stencil points of $(\delta_{zz}\varphi)_{i,j}$ are in the computational domain. The expression for $(\delta_{zz}\varphi)_{i,j}$ is given by (\ref{eq:derivative_wide1}). The truncation error for $(\delta_{zz}\varphi)_{i,j}$ is then
\begin{equation*}
\renewcommand*{\arraystretch}{1.2}
\begin{array}{l}
\varphi_{zz}(\mathbf{x}_{i,j})
- (\delta_{zz}\varphi)_{i,j}
\\
=
\varphi_{zz}(\mathbf{x}_{i,j})
-
\dfrac{
\left.
\mathcal{I}_h \varphi
\right|_{\mathbf{x}_{i,j}+\sqrt{h}(\mathbf{e}_z)_{i,j}}
-2\varphi_{i,j}
+\left.
\mathcal{I}_h \varphi
\right|_{\mathbf{x}_{i,j}-\sqrt{h}(\mathbf{e}_z)_{i,j}}
}
{h}
\\
=
\varphi_{zz}(\mathbf{x}_{i,j})
-
\dfrac{
\varphi(\mathbf{x}_{i,j}+\sqrt{h}(\mathbf{e}_z)_{i,j})
-2\varphi(\mathbf{x}_{i,j})
+\varphi(\mathbf{x}_{i,j}-\sqrt{h}(\mathbf{e}_z)_{i,j})
+O(h^2)
}
{h}
\\
=
O(h)+O(h)
=
O(h).
\end{array}
\end{equation*}
From the first to the second line we have used the fact that the truncation error of the bilinear interpolation is $O(h^2)$.
%and from the second to the last line we have used the quadratic truncation error of central differencing with respect to the wide stencil length $\sqrt{h}$, which is $O\left((\sqrt{h})^2\right)=O(h)$.
%As a comparison, if we choose stencil length $h$ and work through the same algebra, the truncation error will be instead $O(1)$, which is no longer consistent. This justifies the choice of the wide stencil.

Now we consider another case, where one of the stencil points of $(\delta_{zz}\varphi)_{i,j}$ falls outside the computational domain and is thus relocated. Without loss of generality, let us assume again that $\mathbf{x}_{i,j}+\eta_1 (\mathbf{e}_z)_{i,j}\in \partial\Omega$ is the relocated point. The expression for $(\delta_{zz}\varphi)_{i,j}$ is given by (\ref{eq:derivative_wide1_bc1}). The truncation error for $(\delta_{zz}\varphi)_{i,j}$ is then
\begin{equation*}
\renewcommand*{\arraystretch}{1.2}
\begin{array}{l}
\varphi_{zz}(\mathbf{x}_{i,j})
- (\delta_{zz}\varphi)_{i,j}
\\
=
\varphi_{zz}(\mathbf{x}_{i,j})
- \dfrac{
\frac{
\varphi(\mathbf{x}_{i,j}+\eta_1 (\mathbf{e}_z)_{i,j})
-\varphi_{i,j}
}
{\eta_1}
- \frac{
\varphi_{i,j}
- \left.
\mathcal{I}_h \varphi
\right|_{\mathbf{x}_{i,j}-\sqrt{h}(\mathbf{e}_z)_{i,j}}
}
{\sqrt{h}}
}
{
\frac{\eta_1+\sqrt{h}}{2}
}
\\
=
\varphi_{zz}(\mathbf{x}_{i,j})
- \dfrac{
\frac{
\varphi(\mathbf{x}_{i,j}+\eta_1 (\mathbf{e}_z)_{i,j})
-\varphi(\mathbf{x}_{i,j})
}
{\eta_1}
- \frac{
\varphi(\mathbf{x}_{i,j})
- \varphi(\mathbf{x}_{i,j}-\sqrt{h}(\mathbf{e}_z)_{i,j})
}
{\sqrt{h}}
+ O(h^2)
}
{
\frac{\eta_1+\sqrt{h}}{2}
}
\\
=
O(\sqrt{h}-\eta_1)
+ O\left(
\frac{h^2}{\sqrt{h}\frac{\eta_1+\sqrt{h}}{2}}
\right)
=
O(\sqrt{h}).
\end{array}
\end{equation*}

There is one more case, where $\mathbf{x}_{i,j}+\eta_1 (\mathbf{e}_z)_{i,j}\in \partial\Omega$ and $\mathbf{x}_{i,j}-\eta_2 (\mathbf{e}_z)_{i,j}\in \partial\Omega$ are both relocated points. Using the similar argument, one can show that the truncation error for $(\delta_{zz}\varphi)_{i,j}$ is again $O(\sqrt{h})$.

Then, similar to (\ref{eq:localconsistency_narrow}), one can show that the local truncation error of the finite difference scheme at $\mathbf{x}_{i,j}$, where the semi-Lagrangian wide stencil discretization is applied, is given by
\begin{equation}
\label{eq:localconsistency_wide}
\renewcommand*{\arraystretch}{1.2}
\begin{array}{l}
\left|
\;
\mathcal{F}
(\mathbf{x}_{i,j}, \varphi(\mathbf{x}_{i,j}), D^2 \varphi(\mathbf{x}_{i,j}))
- \mathcal{F}_h
(\mathbf{x}_{i,j}, \varphi_h)
\;
\right|
\\
\qquad
= \left\{
\renewcommand*{\arraystretch}{1}
\begin{array}{ll}
O(h),
&
\text{semi-Lagrangian wide stencil, with all the 4}
\\
&
\text{wide stencil points } \in\Omega,
\\
O(\sqrt{h}),
&
\text{semi-Lagrangian wide stencil, otherwise}.
\end{array}
\right.
\end{array}
\end{equation}

Finally, we note that the previous proof has assumed that the optimal control pair is solved exactly, or does not introduce additional truncation error. In Section \ref{sec:policy}, we have mentioned that using linear search for the optimal control pair under the semi-Lagrangian wide stencil discretization introduces truncation error. In particular, if we choose $M=O(N)$, then $O(h)$ truncation error is introduced \cite{wang2008maximal}. As a result, (\ref{eq:localconsistency_wide}) holds.
\hfill
\end{proof}

\subsection{Stability}
\label{subsec:stability}

Another condition for convergence is stability, which means that the discrete system has a bounded solution $u_h$. Stability condition is very closely related to the matrix $\mathbf{A}(a_h,\theta_h)$ in (\ref{eq:discretecomplete}) being an M-matrix \cite{saad2003iterative}, which will be proved in this section. For convenience, given vectors $u_h$ and $v_h$, we use $u_h \geq 0$ and $u_h \geq v_h$ to denote $(u_h)_i \geq 0$ and $(u_h)_i \geq (v_h)_i$ for all $i$. Similarly, given a matrix $\mathbf{A}$, we use $\mathbf{A}\geq 0$ to denote $\mathbf{A}_{ij}\geq 0$ for all $i,j$. In other words, the inequalities for vectors and matrices hold for all the elements.

\begin{lemma}[M-matrix]
\label{lm:WCDDLmatrix}
Suppose an $n\times n$ matrix $\mathbf{A}$ satisfies the following:
\begin{enumerate}
\item $\mathbf{A}$ is an L-matrix: $\mathbf{A}_{ii} > 0$ for all $i$, and $\mathbf{A}_{ij} \leq 0$ for all $i \neq j$;
\item $\mathbf{A}$ is weakly diagonally dominant:
\;
$
| \mathbf{A}_{ii} | \geq
\sum_{j\neq i} | \mathbf{A}_{ij} |
$; and
\item $\mathbf{A}$ has the following connectivity property: Let
$
\mathcal{G}(\mathbf{A})=\left\{ i \left|
| \mathbf{A}_{ii} | >
\sum_{j\neq i} | \mathbf{A}_{ij} |
\right.
\right\}
$
$
\neq \emptyset
$
be the set of rows where strict inequality is achieved.
For any $i\notin \mathcal{G}(\mathbf{A})$, there exists a sequence $i_1, i_2, \cdots, i_k$ with $\mathbf{A}_{i_r,i_{r+1}}\neq 0, \; 0\leq r \leq k-1$, such that $i_0=i$ and $i_k\in \mathcal{G}(\mathbf{A})$.
\end{enumerate}
Then $\mathbf{A}$ is an M-matrix. In particular,
\begin{enumerate}
\item $\mathbf{A}$ is non-singular; and
\item $\mathbf{A}^{-1}\geq 0$, namely, $(\mathbf{A}^{-1})_{ij}\geq 0$ for all $i,j$.
\end{enumerate}
\end{lemma}

\begin{proof}
We refer the readers to \cite{shivakumar1996two,azimzadeh2016weakly,saad2003iterative}.
\hfill
\end{proof}

\begin{lemma}
\label{lm:Mmatrix}
The matrix $\mathbf{A}(a_h,\theta_h)$, defined in (\ref{eq:discretecomplete}), is an M-matrix under the set of admissible controls $(a_h,\theta_h)\in\Gamma$.
\end{lemma}

\begin{proof}
For the matrix $\mathbf{A}(a_h,\theta_h)$, the L-matrix condition and the weakly diagonal dominance condition can be easily verified by checking the four cases in Section \ref{subsec:discrete}. We remark that the strictly diagonally dominant rows correspond to the grid points near the boundary $\partial\Omega$, while the weakly diagonally dominant rows correspond to those inside the computation domain $\Omega$.

The connectivity property of $\mathbf{A}(a_h,\theta_h)$ is yet to be verified. For the grid points $\mathbf{x}_{i,j}$ that are near the boundary, the lexicographical index satisfies $N(i-1)+j\in\mathcal{G}(\mathbf{A})$. For those points that are inside the computational domain, or $N(i-1)+j\notin\mathcal{G}(\mathbf{A})$, there must exist non-zero entries $\mathbf{A}_{N(i-1)+j,N(i'-1)+j'}\neq 0$, where $i'\geq i$, $j'\geq j$, with at lease one strict inequality satisfied. Hence, given any $\mathbf{x}_{i_0,j_0}$, where $N(i_0-1)+j_0\notin\mathcal{G}(\mathbf{A})$, there exist monotonically increasing sequences $i_0 \leq i_1 \leq ... \leq i_k \leq N$ and $j_0 \leq j_1 \leq ... \leq j_k \leq N$, such that $N(i_k-1)+j_k\in\mathcal{G}(\mathbf{A})$.
\hfill
\end{proof}

Before investigating the stability for the nonlinear problem (\ref{eq:discretecomplete}), we first prove the stability for the corresponding linear problem.

\begin{lemma}
\label{lm:stability_inequality}
Define a circle $B_R(0): \{(x,y)|x^2+y^2 \leq R^2\}$, where the radius $R = \displaystyle\max_{(x,y)\in\overline{\Omega}} \sqrt{x^2+y^2}$, such that $B_R(0)$ covers the entire computational domain $\overline{\Omega}$. Let
$
\varphi(\mathbf{x}) \equiv -\frac{1}{2}\|\sqrt{f}\|_\infty (R^2-x^2-y^2)
$
be a lower-bound estimate function that is smooth and non-positive in $\overline{\Omega}$.
Denote its corresponding grid function as $\varphi_h\in \mathbb{R}^{N^2\times 1}$.
Then the vector $\mathbf{A}\varphi_h\in \mathbb{R}^{N^2\times 1}$ satisfies
\begin{equation}
\label{eq:stability_key}
\mathbf{A}\varphi_h \leq -\|\sqrt{f}\|_\infty,
\text{ for all } h.
\end{equation}
\end{lemma}

\begin{proof}
Without loss of generality, let us consider a grid point $\mathbf{x}_{i,j}$  where semi-Lagrangian wide stencil discretization is applied and boundary terms occur with $\mathbf{x}_{i,j}+\sqrt{h}(\mathbf{e}_z)_{i,j}$ relocated to $\mathbf{x}_{i,j}+\eta_1(\mathbf{e}_z)_{i,j}$. Then
\begin{align*}
&
(\mathbf{A}\varphi_h)_{i,j}
= \;
2\left(
\frac{a_{i,j}}{\eta_1 \sqrt{h}} + \frac{1-a_{i,j}}{h}
\right)
\varphi(\mathbf{x}_{i,j})
-\frac{a_{i,j}}{\sqrt{h} \frac{\eta_1+\sqrt{h}}{2}}
\left.
\mathcal{I}_h \varphi
\right|_{\mathbf{x}_{i,j}-\sqrt{h}(\mathbf{e}_z)_{i,j}}
\\
&
\qquad
-\frac{1-a_{i,j}}{h}
\left.
\mathcal{I}_h \varphi
\right|_{\mathbf{x}_{i,j}+\sqrt{h}(\mathbf{e}_w)_{i,j}}
-\frac{1-a_{i,j}}{h}
\left.
\mathcal{I}_h \varphi
\right|_{\mathbf{x}_{i,j}-\sqrt{h}(\mathbf{e}_w)_{i,j}}
\\
&
\leq \;
2\left(
\frac{a_{i,j}}{\eta_1 \sqrt{h}} + \frac{1-a_{i,j}}{h}
\right)
\varphi(\mathbf{x}_{i,j})
-\frac{a_{i,j}}{\eta_1 \frac{\eta_1+\sqrt{h}}{2}}
\varphi(\mathbf{x}_{i,j}+\eta_1 (\mathbf{e}_z)_{i,j})
\\
&
\qquad
-\frac{a_{i,j}}{\sqrt{h} \frac{\eta_1+\sqrt{h}}{2}}
\varphi(\mathbf{x}_{i,j}-\sqrt{h}(\mathbf{e}_z)_{i,j})
-\frac{1-a_{i,j}}{h}
\varphi(\mathbf{x}_{i,j}+\sqrt{h}(\mathbf{e}_w)_{i,j})
\\
&
\qquad
-\frac{1-a_{i,j}}{h}
\varphi(\mathbf{x}_{i,j}-\sqrt{h}(\mathbf{e}_w)_{i,j})
\\
&
= \; -\|\sqrt{f}\|_\infty,
\end{align*}
where we have used $\varphi(\mathbf{x}_{i,j}+\eta_1 (\mathbf{e}_z)_{i,j}) \leq 0$, and
$
\mathcal{I}_h \varphi|_{\mathbf{x}_{i,j}-\sqrt{h}(\mathbf{e}_z)_{i,j}}
\geq \varphi(\mathbf{x}_{i,j}-\sqrt{h}(\mathbf{e}_z)_{i,j})
$
and similarly for the other stencil points. Interested readers can prove the other cases in the same fashion.
\hfill
\end{proof}

\begin{lemma}[Stability for linear problem]
Assume that a control pair $(a,\theta)$ is given, such that the HJB equation (\ref{eq:HJB1}) becomes linear:
\begin{equation*}
\renewcommand*{\arraystretch}{1}
\begin{array}{rl}
- \alpha_{11}(a,\theta) u_{xx}
- 2 \alpha_{12}(a,\theta) u_{xy}
- \alpha_{22}(a,\theta) u_{yy}
=
-2\sqrt{a(1-a)f},
&
\text{ in } \Omega,
\\
u = g,
&
\text{ on } \partial\Omega.
\end{array}
\end{equation*}
Suppose the mixed discretization gives the linear system
$
\mathbf{A}(a_h,\theta_h) \, u_h
= F_h(a_h,\theta_h)
$,
which is the linear version of (\ref{eq:discretecomplete}). Then the solution $u_h$ is bounded as follows:

\begin{enumerate}
\item If $g=0$ (homogeneous boundary condition) and $f\geq 0$ is a bounded function,
\begin{equation}
\label{eq:linearHJB1_discrete_bound}
-\frac{1}{2}\|\sqrt{f}\|_\infty R^2 \leq u_h \leq 0,
\text{ independent of $h$}.
\end{equation}
\item If $f=0$ (homogeneous PDE) and $g$ is a bounded function,
\begin{equation}
\label{eq:linearHJB2_discrete_bound}
\| u_h \|_\infty \leq \|g\|_\infty,
\text{ independent of $h$}.
\end{equation}
\item In general, if $f\geq 0$ and $g$ are bounded functions,
\begin{equation}
\label{eq:linearHJB_discrete_bound}
\| u_h \|_\infty \leq \frac{1}{2}\|\sqrt{f}\|_\infty R^2 + \|g\|_\infty,
\text{ independent of $h$}.
\end{equation}
\end{enumerate}
\end{lemma}

\begin{proof}
1. The proof follows the idea in \cite{samarskii2001theory}. In this case, the $N^2$-vector $F_h$ is simply given by
$F_{i,j} = -2\sqrt{a_{i,j}(1-a_{i,j})f_{i,j}}$.
Since $a_{i,j}\in [0,1]$, we have
$-\|\sqrt{f}\|_\infty \leq F_h \leq 0$.

Lemma \ref{lm:Mmatrix} has proved that $\mathbf{A}$ is an M-matrix, and thus $\mathbf{A}^{-1}\geq 0$. Also, we note that $F_h\leq 0$. Hence, the upper bound of $u_h$ is given by $u_h=\mathbf{A}^{-1}F_h\leq 0$.

Lemma \ref{lm:stability_inequality} has proved that
$\mathbf{A}\varphi_h \leq -\|\sqrt{f}\|_\infty$.
Since $-\|\sqrt{f}\|_\infty \leq F_h = \mathbf{A} u_h$, we have
$\mathbf{A}\varphi_h \leq \mathbf{A} u_h$.
Since $\mathbf{A}^{-1}\geq 0$, we have $\varphi_h \leq u_h$.
Hence, the lower bound of $u_h$ is given by
$
u_h \geq \varphi_h \geq -\|\varphi\|_\infty = -\frac{1}{2}\|\sqrt{f}\|_\infty R^2
$.

2. By Lemma \ref{lm:Mmatrix}, $\mathbf{A}$ is an M-matrix. Then following the proof in \cite{ciarlet1970discrete}, the solution $u_h$ under the M-matrix discretization satisfies the discrete comparison principle, and furthermore, (\ref{eq:linearHJB2_discrete_bound}).

3. This can be obtained by applying the superposition principle of the linear PDEs on 1 and 2.
\hfill
\end{proof}

Eventually, we come back to our original nonlinear problem (\ref{eq:discretecomplete}).

\begin{lemma}[Stability for nonlinear problem]
\label{lm:stability}
Assume that $f$ and $g$ are bounded in $L_\infty$ norm. Given that Lemma \ref{lm:Mmatrix} is satisfied, the solution of the discrete system (\ref{eq:discretecomplete}), $u_h$, is bounded by
\begin{equation}
\label{eq:HJB_discrete_bound}
\|u_h\|_\infty \leq \frac{1}{2}\|\sqrt{f}\|_\infty R^2 + \|g\|_\infty,
\end{equation}
where the bound is independent of the mesh size $h$ and the controls $(a_h,\theta_h)$.
\end{lemma}

\begin{proof}
Since the solution for the linear PDE under the mixed discretization is bounded by
(\ref{eq:linearHJB_discrete_bound}) under all admissible controls $(a_h,\theta_h)\in\Gamma$, and the bound is independent of the controls $(a_h,\theta_h)$ and the mesh size $h$, we conclude that the same bound applies to the solution for the nonlinear PDE under the mixed discretization.
\hfill
\end{proof}

\subsection{Monotonicity}
\label{subsec:monotonicity}

For nonlinear PDEs, monotonicity is another sufficient condition for convergence in the viscosity sense. Monotonicity means that the discretization scheme at a grid point $\mathbf{x}_{i,j}$ must be a non-decreasing function of the unknown $u_{i,j}$ and a non-increasing function of the unknowns at the other points $\{u_{p,q}|_{(p,q)\neq (i,j)}\}$. Monotonicity of our numerical scheme (\ref{eq:discretecomplete1})-(\ref{eq:discretecomplete2}) is inherited from the M-matrix property in Lemma \ref{lm:WCDDLmatrix}.

\begin{lemma}[Monotonicity]
The finite difference discretization
$
\mathcal{F}_h
( \mathbf{x}_{i,j}, u_h )
= \mathcal{F}_h
( \mathbf{x}_{i,j}, u_{i,j}, \{u_{p,q}|_{(p,q)\neq (i,j)}\} ) = 0,
$
given in (\ref{eq:discretecomplete1})-(\ref{eq:discretecomplete2}), is monotone.
More specifically, for all $u_h \leq v_h$, we have
\begin{equation}
\label{eq:monotonicity}
\renewcommand*{\arraystretch}{1.2}
\begin{array}{l}
\mathcal{F}_h
( \mathbf{x}_{i,j}, u_{i,j}, \{u_{p,q}|_{(p,q)\neq (i,j)}\} )
\leq
\mathcal{F}_h
( \mathbf{x}_{i,j}, v_{i,j}, \{u_{p,q}|_{(p,q)\neq (i,j)}\} ),
\\
\mathcal{F}_h
( \mathbf{x}_{i,j}, u_{i,j}, \{u_{p,q}|_{(p,q)\neq (i,j)}\} )
\geq
\mathcal{F}_h
( \mathbf{x}_{i,j}, u_{i,j}, \{v_{p,q}|_{(p,q)\neq (i,j)}\} ).
\end{array}
\end{equation}
\end{lemma}

\begin{proof}
The proof follows \cite{forsyth2007numerical}. Our goal is to verify the monotonicity condition (\ref{eq:monotonicity}). Without loss of generality, let us analyze one example: $u_h \leq v_h$ with $u_{i,j}= v_{i,j}$. Then
\begin{align*}
&
\mathcal{F}_h
( \mathbf{x}_{i,j}, u_{i,j}, \{u_{p,q}|_{(p,q)\neq (i,j)}\} )
-
\mathcal{F}_h
( \mathbf{x}_{i,j}, u_{i,j}, \{v_{p,q}|_{(p,q)\neq (i,j)}\} )
\\
=
&
\displaystyle\max_{
(a_{i,j},\theta_{i,j})
\in\Gamma
}
\left\{
(\mathbf{A} (a_{i,j},\theta_{i,j}) \, u_h)_{i,j}
- F_{i,j}(a_{i,j},\theta_{i,j})
\right\}
\\
&
\hspace{3cm}
- \displaystyle\max_{
(a_{i,j},\theta_{i,j})
\in\Gamma
}
\left\{
(\mathbf{A} (a_{i,j},\theta_{i,j}) \, v_h)_{i,j}
- F_{i,j}(a_{i,j},\theta_{i,j})
\right\}
\\
\geq
&
\displaystyle\min_{
(a_{i,j},\theta_{i,j})
\in\Gamma
}
[ (\mathbf{A} (a_{i,j},\theta_{i,j}) (u_h - v_h) ]_{i,j}
\;
\geq
\;
0,
\end{align*}
where the first inequality uses
$
\displaystyle\max_{x} f(x) - \displaystyle\max_{x} g(x)
\geq
\displaystyle\min_{x} \left[ f(x) - g(x) \right]
$,
and the last inequality considers that $u_h - v_h \leq 0$ and that all the off-diagonal entries of $\mathbf{A}$ are non-positive under all admissible controls.
\hfill
\end{proof}

\subsection{Strong comparison principle}
\label{subsec:comparison}

There is one more sufficient condition for convergence, called strong comparison principle \cite{barles1991convergence}. Strong comparison principle holds if the boundary condition is satisfied in the viscosity sense. Unfortunately, there is no proof in the literature that this necessarily holds for the Dirichlet problem (\ref{eq:MAE2}). Hence, we provide a proof in the setting of our proposed numerical scheme.

\begin{lemma}
\label{lm:comparison_max}
Let
$
\zeta(\mathbf{x};\mathbf{p})
\equiv \frac{1}{2} \|\sqrt{f}\|_\infty \|\mathbf{x}-\mathbf{p}\|_2^2
$,
where $\mathbf{p}\in \mathbb{R}^2$ is a random vector.
Let
$\hat{u}(\mathbf{x}): \{\mathbf{x}_{i,j}\in\Omega\} \cup \partial\Omega \to \mathbb{R}$,
where
$
\hat{u}(\mathbf{x}) \equiv \left\{
\begin{array}{ll}
u_h(\mathbf{x}_{i,j}),
& \text{ if } \mathbf{x}\in\{\mathbf{x}_{i,j}\in\Omega\},
\\
g(\mathbf{x}),
& \text{ if } \mathbf{x}\in\partial\Omega.
\end{array}
\right.
$
Then $\mathcal{I}_h\zeta \pm \hat{u}$ achieves its maximum on $\partial \Omega$.
\end{lemma}

\begin{proof}
Without loss of generality, let us consider again a grid point $\mathbf{x}_{i,j}\notin \partial\Omega$ where semi-Lagrangian wide stencil discretization is applied and boundary terms occur with $\mathbf{x}_{i,j}+\sqrt{h}(\mathbf{e}_z)_{i,j}$ relocated to $\mathbf{x}_{i,j}+\eta_1(\mathbf{e}_z)_{i,j}$. Assume that the control pair is fixed. Define a linear stencil operator on an arbitrary function $u$ at $\mathbf{x}_{i,j}$ as
\begin{equation*}
\renewcommand*{\arraystretch}{1.2}
\begin{array}{rl}
\mathcal{S}[u](\mathbf{x}_{i,j})
&
\equiv
2\left(
\frac{a_{i,j}}{\eta_1 \sqrt{h}} + \frac{1-a_{i,j}}{h}
\right)
u|_{\mathbf{x}_{i,j}}
-\frac{a_{i,j}}{\sqrt{h} \frac{\eta_1+\sqrt{h}}{2}}
u|_{\mathbf{x}_{i,j}-\sqrt{h}(\mathbf{e}_z)_{i,j}}
\\
&
-\frac{a_{i,j}}{\eta_1 \frac{\eta_1+\sqrt{h}}{2}}
u|_{\mathbf{x}_{i,j}+\eta_1 (\mathbf{e}_z)_{i,j}}
-\frac{1-a_{i,j}}{h}
u|_{\mathbf{x}_{i,j}+\sqrt{h}(\mathbf{e}_w)_{i,j}}
-\frac{1-a_{i,j}}{h}
u|_{\mathbf{x}_{i,j}-\sqrt{h}(\mathbf{e}_w)_{i,j}}.
\end{array}
\end{equation*}
We note that the relocated stencil point is also included in the operator. Then we have
$\mathcal{S}[\mathcal{I}_h\zeta](\mathbf{x}_{i,j})
\leq \mathcal{S}[\zeta](\mathbf{x}_{i,j})
= -\|\sqrt{f}\|_\infty$,
and
$\mathcal{S}[\hat{u}](\mathbf{x}_{i,j})
= -2\sqrt{a_{i,j}(1-a_{i,j})f_{i,j}}$.
As a result, we have
$\mathcal{S}[\mathcal{I}_h\zeta \pm \hat{u}](\mathbf{x}_{i,j})
= -\|\sqrt{f}\|_\infty \pm 2\sqrt{a_{i,j}(1-a_{i,j})f_{i,j}} \leq 0$.

Now assume that $\mathcal{I}_h\zeta \pm \hat{u}$ achieves its maximum at this grid point $\mathbf{x}_{i,j}$. Next we prove that
$
(\mathcal{I}_h\zeta \pm \hat{u})|_{\mathbf{y}}
= (\mathcal{I}_h\zeta \pm \hat{u})|_{\mathbf{x}_{i,j}}
$
for any stencil point $\mathbf{y}$ connected to $\mathbf{x}_{i,j}$, namely, for any
$\mathbf{y}\in \{
\mathbf{x}_{i,j}+\eta_1 (\mathbf{e}_z)_{i,j},
\mathbf{x}_{i,j}-\sqrt{h}(\mathbf{e}_z)_{i,j},
\mathbf{x}_{i,j}\pm\sqrt{h}(\mathbf{e}_w)_{i,j}
\}$.
This can be proved by contradiction. Assume that there exists at least one stencil point where the strict inequality holds, namely,
$
(\mathcal{I}_h\zeta \pm \hat{u})|_{\mathbf{y}}
< (\mathcal{I}_h\zeta \pm \hat{u})|_{\mathbf{x}_{i,j}}
$.
Then
\begin{equation*}
\renewcommand*{\arraystretch}{1.2}
\begin{array}{rl}
\mathcal{S}[\mathcal{I}_h\zeta \pm \hat{u}](\mathbf{x}_{i,j})
&
> \left[
2\left(
\frac{a_{i,j}}{\eta_1 \sqrt{h}} + \frac{1-a_{i,j}}{h}
\right)
\right.
\\
&
\left.
-\frac{a_{i,j}}{\sqrt{h} \frac{\eta_1+\sqrt{h}}{2}}
-\frac{a_{i,j}}{\eta_1 \frac{\eta_1+\sqrt{h}}{2}}
-\frac{1-a_{i,j}}{h}
-\frac{1-a_{i,j}}{h}
\right]
(\mathcal{I}_h\zeta \pm \hat{u})|_{\mathbf{x}_{i,j}}
= 0,
\end{array}
\end{equation*}
which contradicts with 
$\mathcal{S}[\mathcal{I}_h\zeta \pm \hat{u}](\mathbf{x}_{i,j}) \leq 0$.
The key point of this result is that
$
(\mathcal{I}_h\zeta \pm \hat{u})|_{\mathbf{x}_{i,j}+\eta_1 (\mathbf{e}_z)_{i,j}}
= (\mathcal{I}_h\zeta \pm \hat{u})|_{\mathbf{x}_{i,j}}
$.
That is, $\mathcal{I}_h\zeta \pm \hat{u}$ achieves its maximum at the boundary point $\mathbf{x}_{i,j}+\eta_1 (\mathbf{e}_z)_{i,j} \in \partial\Omega$.

In general, consider any grid point $\mathbf{x}_{i,j}\notin\partial \Omega$. Assume that $\mathcal{I}_h\zeta \pm \hat{u}$ achieves its maximum at $\mathbf{x}_{i,j}$. One can prove in the same fashion that
$
(\mathcal{I}_h\zeta \pm \hat{u})|_{\mathbf{y}}
= (\mathcal{I}_h\zeta \pm \hat{u})|_{\mathbf{x}_{i,j}}
$
for any stencil point $\mathbf{y}$ connected to $\mathbf{x}_{i,j}$. Then by the connectivity property (see the proof of Lemma \ref{lm:Mmatrix}), there exists a boundary point $\mathbf{z}\in\partial\Omega$, such that
$
(\mathcal{I}_h\zeta \pm \hat{u})|_{\mathbf{z}}
= (\mathcal{I}_h\zeta \pm \hat{u})|_{\mathbf{x}_{i,j}}
$.
Hence, $\mathcal{I}_h\zeta \pm \hat{u}$ achieves its maximum at the boundary point $\mathbf{z}\in \partial \Omega$.
\hfill
\end{proof}

\begin{lemma}
\label{lm:comparison_bc}
Let $\Omega$ be a strictly convex domain. Assume that Lemma \ref{lm:comparison_max} holds. Define
\begin{equation*}
\overline{u}(\mathbf{x})
\equiv \limsup_{\substack{h\to 0, \, \mathbf{y}\to \mathbf{x}}} u_h(\mathbf{y}),
\quad
\underline{u}(\mathbf{x})
\equiv \liminf_{\substack{h\to 0, \, \mathbf{y}\to \mathbf{x}}} u_h(\mathbf{y}).
\end{equation*}
Then $\overline{u}(\mathbf{x})=\underline{u}(\mathbf{x})=g(\mathbf{x})$ for all $\mathbf{x}\in\partial \Omega$.
\end{lemma}

\begin{proof}
Once Lemma \ref{lm:comparison_max} holds, the proof follows Lemma 6.4 in \cite{feng2016convergent}.
\end{proof}

Lemma \ref{lm:comparison_bc} is essentially the comparison result on the boundary $\partial\Omega$. Now we are ready to extend the comparison result to the entire computational domain $\overline{\Omega}$.

\begin{lemma}
\label{lm:comparison_subsup}
Given that the finite difference discretization (\ref{eq:discretecomplete1})-(\ref{eq:discretecomplete2}) satisfies consistency, stability and monotonicity, $\overline{u}(\mathbf{x})$ and $\underline{u}(\mathbf{x})$ are respectively the viscosity subsolution and supersolution of the Dirichlet problem (\ref{eq:MAE2}).
\end{lemma}

\begin{proof}
See the proof of Theorem 2.1 in \cite{barles1991convergence}.
\end{proof}

\begin{lemma}[Strong comparison principle]
Let $\Omega$ be a strictly convex domain. Then the finite difference discretization (\ref{eq:discretecomplete1})-(\ref{eq:discretecomplete2}) satisfies $\overline{u}\leq \underline{u}$ in $\overline{\Omega}$.
\end{lemma}

\begin{proof}
Since $\overline{u}$ and $\underline{u}$ are respectively the viscosity subsolution and supersolution (Lemma \ref{lm:comparison_subsup}), and $\overline{u}\leq \underline{u}$ on $\partial\Omega$ (Lemma \ref{lm:comparison_bc}), by Theorem 3.3 in \cite{crandall1992user}, we conclude that $\overline{u}\leq \underline{u}$ in $\overline{\Omega}$.
\end{proof}

\subsection{Convergence of the numerical solution to the viscosity solution}

Once consistency, stability, monotonicity and strong comparison principle are proved, Barles-Souganidis theorem \cite{barles1991convergence} guarantees the convergence of the numerical solution to the viscosity solution.

\begin{theorem}[Barles-Souganidis theorem]
Let $\Omega$ be a strictly convex domain. Given that the finite difference discretization (\ref{eq:discretecomplete1})-(\ref{eq:discretecomplete2}) satisfies consistency, stability, monotonicity and strong comparison principle, the numerical solution converges to the viscosity solution of the Dirichlet problem (\ref{eq:MAE2}).
\end{theorem}

\begin{proof}
See Barles and Souganidis's proof of Theorem 2.1 in \cite{barles1991convergence}.
\hfill
\end{proof}

%% Numerical Results
%% ==============

\section{Numerical Results}
\label{sec:numerical}

In this section, we will present numerical results for the Monge-Amp\`ere equation using our proposed mixed standard 7-point stencil and semi-Lagrangian wide stencil scheme. These numerical results show that the mixed scheme can achieve second order convergence rate whenever the standard 7-point stencils can be applied monotonically on the entire computational domain, and up to order one convergence rate otherwise. Compared to the pure semi-Lagrangian wide stencil scheme in \cite{feng2016convergent}, our proposed mixed scheme yields a smaller discretization error $\|u-u_h\|$ and a faster convergence rate. The examples we consider in this section come from \cite{froese2011convergent,benamou2010two}. We choose the tolerance of residual for the policy iteration to be $10^{-6}$. We let the initial guess of the numerical solution be the solution of
\begin{equation}
\renewcommand*{\arraystretch}{1}
\begin{array}{rll}
u_{xx} + u_{yy}
&
= 2\sqrt{f},
&
\quad
\text{ in } \Omega,
\\
u
&
= g, 
&
\quad
\text{ on } \partial\Omega,
\end{array}
\end{equation}
which corresponds to the solution of (\ref{eq:HJB1}) with $a=\frac{1}{2}$ and arbitrary $\theta$. We choose the grid size $N^2 = 32^2,64^2,\cdots,512^2$, and define the numerical convergence rate as
$
\log_2
\frac{\|u-u_h(\frac{N}{2})\|}
{\|u-u_h(N)\|}
$,
where $u_h(N)$ is the numerical solution on an $N\times N$ grid.

\begin{figure}[b!]
\begin{center}
\includegraphics[scale=0.35]{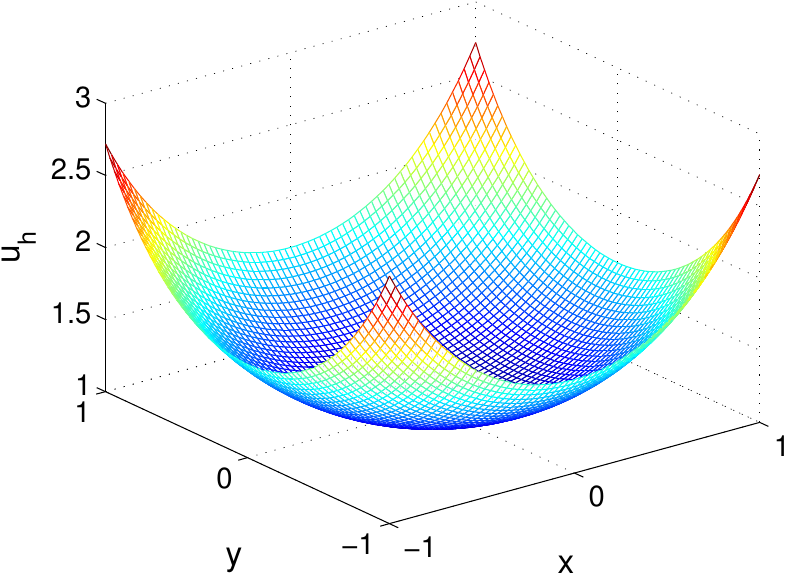}
\includegraphics[scale=0.4]{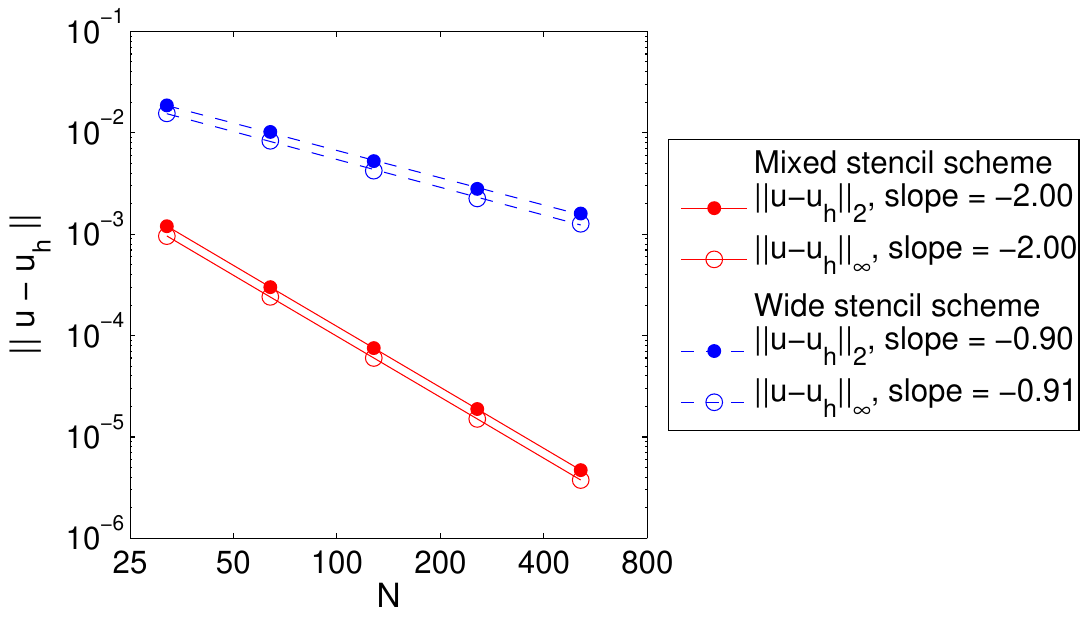}
\end{center}
\vspace{-3mm}
\hspace{2.2cm}
(1)
\hspace{5cm}
(2)
\caption{\label{fig:example1}
Numerical results of Example 1, where the exact solution is $u(x,y) = e^{\frac{1}{2}(x^2+y^2)}$.
(1) Numerical solution.
(2) Norms of the errors $\|u-u_h\|$. For the proposed mixed stencil scheme (red-solid), the convergence rates, indicated by the slopes, are $O(h^2)$ in both $L_2$ and $L_\infty$ norms. For the pure semi-Lagrangian wide stencil scheme (blue-dashed), the convergence rates are approximately $O(h)$ in both $L_2$ and $L_\infty$ norms.}
\end{figure}

\begin{table}[b!]
\footnotesize
\begin{center}
(1) Proposed mixed stencil scheme
\begin{tabular}{|c||c|c||c|c||c|}
\hline
$N$
& $\|u-u_h\|_2$
& \begin{tabular}{c}
Numerical
\\
convergence
\\
rate
\end{tabular}
& $\|u-u_h\|_\infty$
& \begin{tabular}{c}
Numerical
\\
convergence
\\
rate
\end{tabular}
& \begin{tabular}{c}
Number of
\\
policy
\\
iterations
\end{tabular}
\\
\hline
%16
%& 4.769$\times 10^{-3}$
%& 
%& 3.810$\times 10^{-3}$
%& 
%& 4
%\\
32
& 1.201$\times 10^{-3}$
& %1.99
& 9.598$\times 10^{-4}$
& %1.99
& 4
\\
64
& 3.009$\times 10^{-4}$
& 2.00
& 2.404$\times 10^{-4}$
& 2.00
& 4
\\
128
& 7.526$\times 10^{-5}$
& 2.00
& 6.013$\times 10^{-5}$
& 2.00
& 4
\\
256
& 1.882$\times 10^{-5}$
& 2.00
& 1.504$\times 10^{-5}$
& 2.00
& 4
\\
512
& 4.705$\times 10^{-6}$
& 2.00
& 3.759$\times 10^{-6}$
& 2.00
& 4
\\
\hline
\end{tabular}

\smallskip
(2) Pure semi-Lagrangian wide stencil scheme
\begin{tabular}{|c||c|c||c|c||c|}
\hline
$N$
& $\|u-u_h\|_2$
& \begin{tabular}{c}
Numerical
\\
convergence
\\
rate
\end{tabular}
& $\|u-u_h\|_\infty$
& \begin{tabular}{c}
Numerical
\\
convergence
\\
rate
\end{tabular}
& \begin{tabular}{c}
Number of
\\
policy
\\
iterations
\end{tabular}
\\
\hline
%16
%& 2.651$\times 10^{-2}$
%& 
%& 2.398$\times 10^{-2}$
%& 
%& 5
%\\
32
& 1.868$\times 10^{-2}$
& %0.50
& 1.557$\times 10^{-2}$
& %0.62
& 5
\\
64
& 1.020$\times 10^{-2}$
& 0.87
& 8.364$\times 10^{-3}$
& 0.90
& 5
\\
128
& 5.263$\times 10^{-3}$
& 0.95
& 4.240$\times 10^{-3}$
& 0.98
& 6
\\
256
& 2.801$\times 10^{-3}$
& 0.91
& 2.259$\times 10^{-3}$
& 0.91
& 5
\\
512
& 1.600$\times 10^{-3}$
& 0.81
& 1.268$\times 10^{-3}$
& 0.83
& 5
\\
\hline
\end{tabular}
\end{center}
\smallskip
\caption{\label{tab:example1}
Numerical results of Example 1, where the exact solution is $u(x,y) = e^{\frac{1}{2}(x^2+y^2)}$.
(1) Proposed mixed stencil scheme. The convergence rates in both $L_2$ and $L_\infty$ norms are $O(h^2)$.
(2) Pure semi-Lagrangian wide stencil scheme. The convergence rates in both $L_2$ and $L_\infty$ norms are approximately $O(h)$.}
\end{table}

\underline{Example 1.}
Start with
\begin{equation*}
\displaystyle
f(x,y) = (1+x^2+y^2)e^{x^2+y^2},
\qquad
\displaystyle
g(x,y) = e^{\frac{1}{2}(x^2+y^2)},
\qquad
\overline{\Omega} = [-1,1]\times[-1,1],
\end{equation*}
where the exact solution
$
u(x,y) = e^{\frac{1}{2}(x^2+y^2)}
$
is smooth. For this example, it turns out that the standard 7-point stencil discretization can be applied on the entire computational domain and still results in a monotone scheme, since the optimal control pair $(a^*,\theta^*)$ at every grid point is inside the 7-point-stencil regions $\Gamma^1\cup\Gamma^2\cup\partial\Gamma^0$.
Consequentially, the numerical solution converges at the optimal theoretical convergence rate $O(h^2)$; see Figure \ref{fig:example1}(2,red-solid) and Table \ref{tab:example1}(1). We observe that the computation is efficient, in the sense that the number of policy iterations remains a small constant 4 as $N$ increases.

We compare the proposed mixed scheme with the pure semi-Lagrangian wide stencil scheme in \cite{feng2016convergent}, where the wide stencils are applied on the entire computation domain. Figure \ref{fig:example1}(2,blue-dashed) and Table \ref{tab:example1}(2) show that the convergence rate of the pure wide stencil scheme is approximately first order. We note that order one is the optimal theoretical convergence rate for the pure wide stencil scheme; see Lemma \ref{lm:localconsistency}. The convergence rate using the proposed mixed scheme is significantly faster than the rate using the pure semi-Lagrangian wide stencil scheme.

\begin{figure}[b!]
\begin{center}
\includegraphics[scale=0.35]{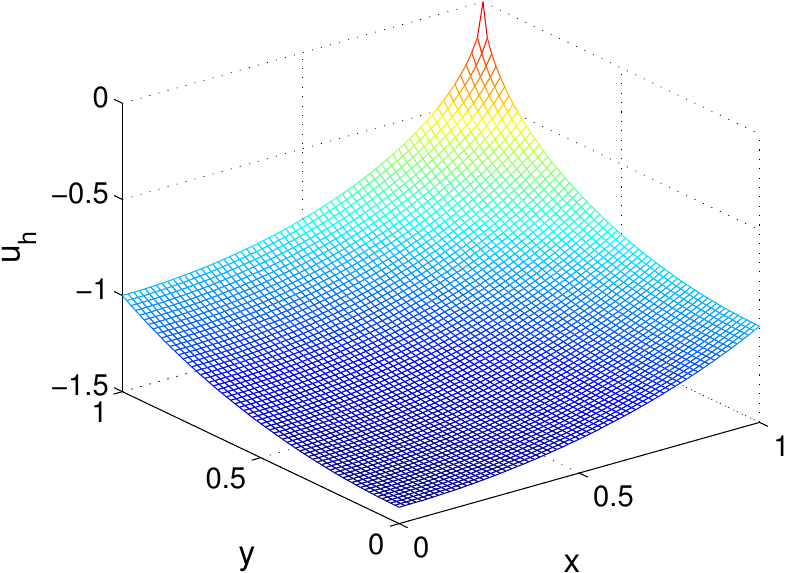}
\includegraphics[scale=0.4]{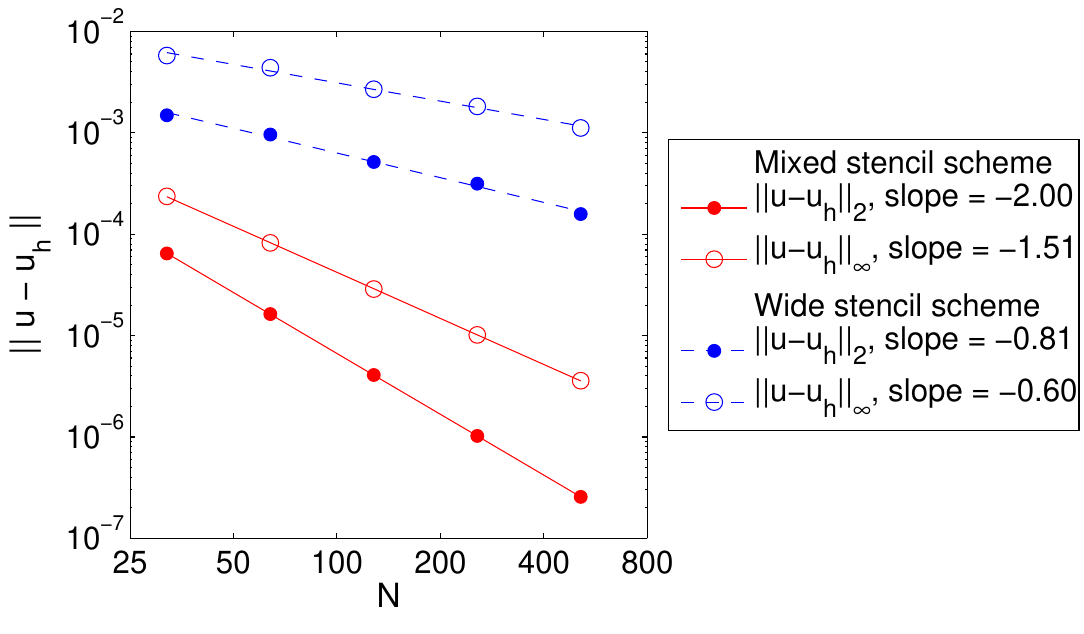}
\end{center}
\vspace{-3mm}
\hspace{2.2cm}
(1)
\hspace{5cm}
(2)
\caption{\label{fig:example2}
Numerical results of Example 2, where the exact solution is $u(x,y) = -\sqrt{2-x^2-y^2}$.
(1) Numerical solution.
(2) Norms of the errors $\|u-u_h\|$. For the proposed mixed stencil scheme (red-solid), the convergence rates, indicated by the slopes, are $O(h^2)$ in $L_2$ norm and $O(h^{1.5})$ in $L_\infty$ norm, respectively. For the pure semi-Lagrangian wide stencil scheme (blue-dashed), the convergence rates are worse than $O(h)$ in both $L_2$ and $L_\infty$ norms.}
\end{figure}

\begin{table}[b!]
\footnotesize
\begin{center}
(1) Proposed mixed stencil scheme
\begin{tabular}{|c||c|c||c|c||c|}
\hline
$N$
& $\|u-u_h\|_2$
& \begin{tabular}{c}
Numerical
\\
convergence
\\
rate
\end{tabular}
& $\|u-u_h\|_\infty$
& \begin{tabular}{c}
Numerical
\\
convergence
\\
rate
\end{tabular}
& \begin{tabular}{c}
Number of
\\
policy
\\
iterations
\end{tabular}
\\
\hline
%16
%& 2.506$\times 10^{-4}$
%& 
%& 6.898$\times 10^{-4}$
%& 
%& 4
%\\
32
& 6.450$\times 10^{-5}$
& %1.96
& 2.359$\times 10^{-4}$
& %1.55
& 4
\\
64
& 1.628$\times 10^{-5}$
& 1.99
& 8.211$\times 10^{-5}$
& 1.52
& 5
\\
128
& 4.084$\times 10^{-6}$
& 2.00
& 2.882$\times 10^{-5}$
& 1.51
& 5
\\
256
& 1.022$\times 10^{-6}$
& 2.00
& 1.015$\times 10^{-5}$
& 1.51
& 5
\\
512
& 2.557$\times 10^{-7}$
& 2.00
& 3.583$\times 10^{-6}$
& 1.50
& 5
\\
\hline
\end{tabular}

\smallskip
(2) Pure semi-Lagrangian wide stencil scheme
\begin{tabular}{|c||c|c||c|c||c|}
\hline
$N$
& $\|u-u_h\|_2$
& \begin{tabular}{c}
Numerical
\\
convergence
\\
rate
\end{tabular}
& $\|u-u_h\|_\infty$
& \begin{tabular}{c}
Numerical
\\
convergence
\\
rate
\end{tabular}
& \begin{tabular}{c}
Number of
\\
policy
\\
iterations
\end{tabular}
\\
\hline
%16
%& 2.728$\times 10^{-3}$
%& 
%& 8.745$\times 10^{-3}$
%& 
%& 4
%\\
32
& 1.493$\times 10^{-3}$
& %?
& 5.799$\times 10^{-3}$
& %?
& 5
\\
64
& 9.634$\times 10^{-4}$
& 0.63
& 4.394$\times 10^{-3}$
& 0.40
& 4
\\
128
& 5.166$\times 10^{-4}$
& 0.90
& 2.697$\times 10^{-3}$
& 0.70
& 5
\\
256
& 3.153$\times 10^{-4}$
& 0.71
& 1.824$\times 10^{-3}$
& 0.56
& 5
\\
512
& 1.583$\times 10^{-4}$
& 0.99
& 1.120$\times 10^{-3}$
& 0.70
& 5
\\
\hline
\end{tabular}
\end{center}
\smallskip
\caption{\label{tab:example2}
Numerical results of Example 2, where the exact solution is $u(x,y) = -\sqrt{2-x^2-y^2}$.
(1) Proposed mixed stencil scheme. The convergence rates in $L_2$ and $L_\infty$ norms are $O(h^2)$ and $O(h^{1.5})$, respectively.
(2) Pure semi-Lagrangian wide stencil scheme. The convergence rates in both $L_2$ and $L_\infty$ norms are worse than $O(h)$.}
\end{table}

\underline{Example 2.}
Consider
\begin{equation*}
f(x,y) = \frac{2}{(2-x^2-y^2)^2},
\qquad
g(x,y) = -\sqrt{2-x^2-y^2},
\qquad
\overline{\Omega} = [0,1]\times[0,1],
\end{equation*}
where $f$ is singular at $(1,1)$, and the exact solution is
$
u(x,y) = -\sqrt{2-x^2-y^2}
$.
Similar to Example 1, we can apply the standard 7-point stencil discretization monotonically on the entire $\Omega$. The convergence rates are $O(h^2)$ and $O(h^{1.5})$ in $L_2$ and $L_\infty$ norms respectively; see Figure \ref{fig:example2}(2,red-solid) and Table \ref{tab:example2}(1). As a comparison, if we applied the pure semi-Lagrangian wide stencil scheme, then the convergence rate is worse than $O(h)$; see Figure \ref{fig:example2}(2,blue-dashed) and Table \ref{tab:example2}(2).

\begin{figure}[b!]
\begin{center}
\includegraphics[scale=0.35]{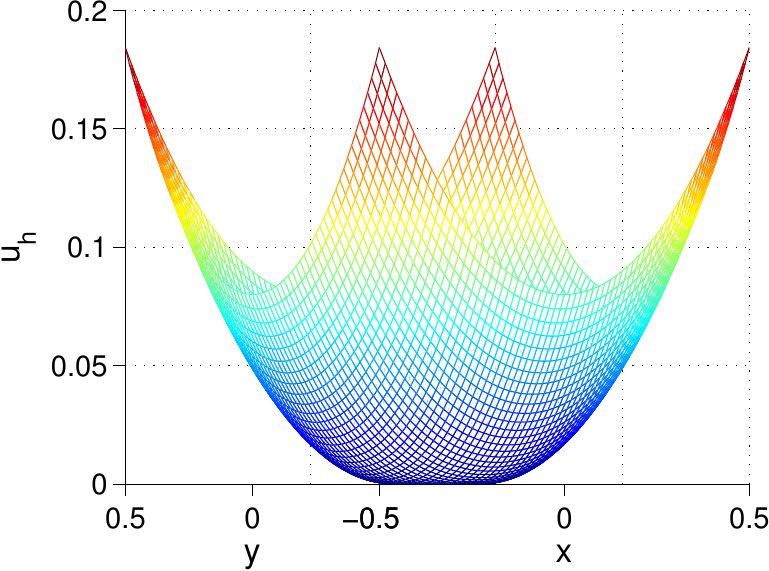}
\includegraphics[scale=0.4]{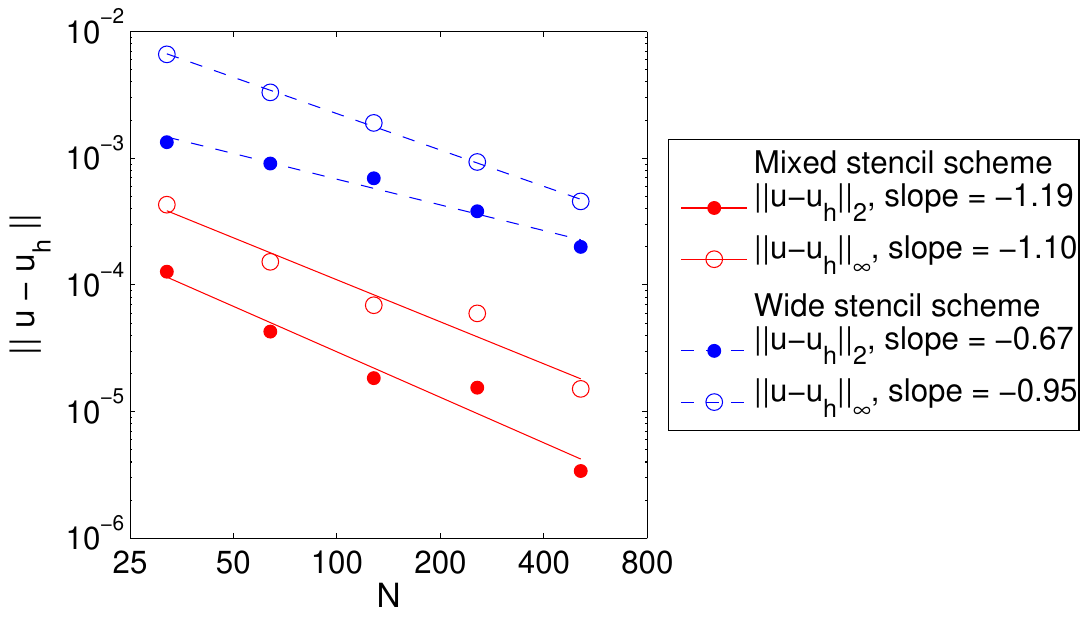}
\end{center}
\vspace{-3mm}
\hspace{2.2cm}
(1)
\hspace{5cm}
(2)
\caption{\label{fig:example3}
Numerical results of Example 3, where the exact solution is $\frac{1}{2}\max\left( \sqrt{x^2+y^2}- 0.1, \right.$ $\left. 0 \right)^2$.
(1) Numerical solution.
(2) Norms of the error $\|u-u_h\|$. For the proposed mixed stencil scheme (red-solid), the convergence rates, indicated by the slopes, are approximately $O(h)$ in both $L_2$ and $L_\infty$ norms. For the pure semi-Lagrangian wide stencil scheme (blue-dashed), the errors are larger than the mixed scheme, and the convergence rates are worse than $O(h)$ in both $L_2$ and $L_\infty$ norms.}
\end{figure}

\begin{table}[t!]
\footnotesize
\begin{center}
(1) Proposed mixed stencil scheme
\begin{tabular}{|c||c|c||c|c||c|}
\hline
$N$
& $\|u-u_h\|_2$
& \begin{tabular}{c}
Numerical
\\
convergence
\\
rate
\end{tabular}
& $\|u-u_h\|_\infty$
& \begin{tabular}{c}
Numerical
\\
convergence
\\
rate
\end{tabular}
& \begin{tabular}{c}
Number of
\\
policy
\\
iterations
\end{tabular}
\\
\hline
%16
%& 2.499$\times 10^{-4}$
%& 
%& 8.528$\times 10^{-4}$
%& 
%& 3
%\\
32
& 1.270$\times 10^{-4}$
& %0.98
& 4.298$\times 10^{-4}$
& %0.99
& 4
\\
64
& 4.273$\times 10^{-5}$
& 1.57
& 1.520$\times 10^{-4}$
& 1.50
& 6
\\
128
& 1.835$\times 10^{-5}$
& 1.22
& 6.907$\times 10^{-5}$
& 1.14
& 7
\\
256
& 1.544$\times 10^{-5}$
& 0.25
& 5.959$\times 10^{-5}$
& 0.21
& 9
\\
512
& 3.396$\times 10^{-6}$
& 2.18
& 1.513$\times 10^{-5}$
& 1.98
& 20
\\
\hline
\end{tabular}

\smallskip
(2) Pure semi-Lagrangian wide stencil scheme
\begin{tabular}{|c||c|c||c|c||c|}
\hline
$N$
& $\|u-u_h\|_2$
& \begin{tabular}{c}
Numerical
\\
convergence
\\
rate
\end{tabular}
& $\|u-u_h\|_\infty$
& \begin{tabular}{c}
Numerical
\\
convergence
\\
rate
\end{tabular}
& \begin{tabular}{c}
Number of
\\
policy
\\
iterations
\end{tabular}
\\
\hline
%16
%& 
%& 
%& 
%& 
%& 
%\\
32
& 1.337$\times 10^{-3}$
& %?
& 6.604$\times 10^{-3}$
& %?
& 5
\\
64
& 9.084$\times 10^{-4}$
& 0.56
& 3.304$\times 10^{-3}$
& 1.00
& 6
\\
128
& 6.940$\times 10^{-4}$
& 0.39
& 1.901$\times 10^{-3}$
& 0.80
& 7
\\
256
& 3.815$\times 10^{-4}$
& 0.86
& 9.335$\times 10^{-4}$
& 1.03
& 7
\\
512
& 1.998$\times 10^{-4}$
& 0.93
& 4.563$\times 10^{-4}$
& 1.03
& 9
\\
\hline
\end{tabular}
\end{center}
\smallskip
\caption{\label{tab:example3}
Numerical results for Example 3, where the exact solution is $\frac{1}{2}\max\left( \sqrt{x^2+y^2}- 0.1, 0 \right)^2$.
(1) Proposed mixed stencil scheme.
(2) Pure semi-Lagrangian wide stencil scheme. The errors $\|u-u_h\|$ by the proposed mixed stencil scheme are smaller than those by the pure wide stencil scheme.}
\end{table}

\underline{Example 3.}
Consider
\begin{equation*}
\renewcommand*{\arraystretch}{1}
\begin{array}{r}
f(x,y) = \max\left( 1- \dfrac{0.1}{\sqrt{x^2+y^2}}, 0 \right),
\qquad
g(x,y) = \dfrac{1}{2}( \sqrt{x^2+y^2} - 0.1 )^2,
\\
\overline{\Omega} = [-0.5,0.5]\times[-0.5,0.5].
\end{array}
\end{equation*}
The exact solution is given by
$
u(x,y) = \frac{1}{2}\max\left( \sqrt{x^2+y^2} - 0.1, 0 \right)^2
$.
This is a $C^1$ function where the singularity occurs at the ring $x^2+y^2=0.1^2$. First we consider the proposed mixed scheme. Semi-Lagrangian wide stencils need to be applied near the ring $x^2+y^2=0.1^2$. Figure \ref{fig:example3}(2,red-solid) and Table \ref{tab:example3}(1) show the numerical results. We note that the error reduction rates for the sequence of $N=32,64,\cdots,512$ do not look as regular as the previous examples. The reason is that wide stencil introduces interpolation error, which fluctuates as $N$ increases, despite converging towards 0. However, a clear error reduction, and thus convergence, can be observed.
For comparison, we also test the pure semi-Lagrangian wide stencil scheme, as shown in Figure \ref{fig:example3}(2,blue-dashed) and Table \ref{tab:example3}(2). Our proposed mixed scheme performs better than the pure wide stencil scheme, in the sense that the error $\|u-u_h\|$ is significantly smaller, and the convergence rate is faster.

\begin{figure}[t!]
\begin{center}
\includegraphics[scale=0.35]{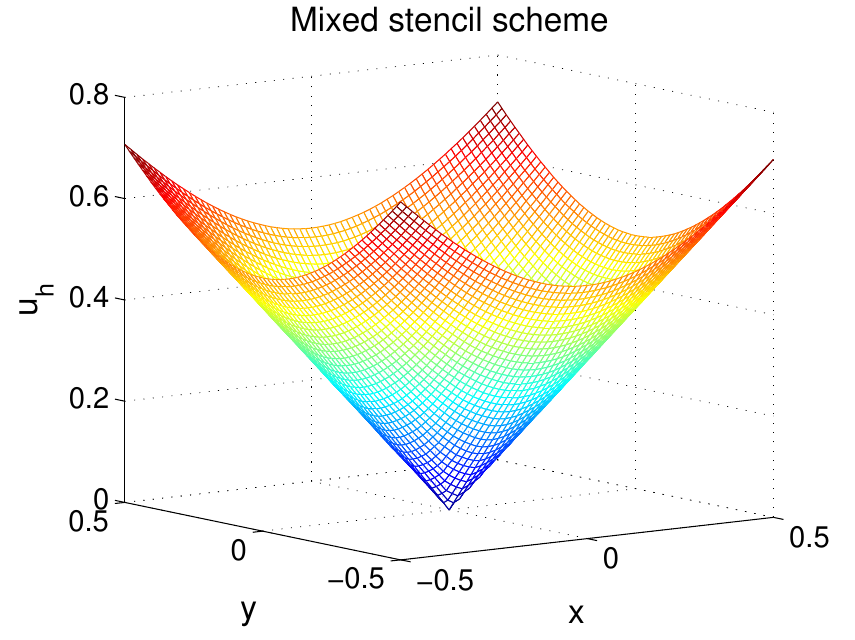}
\hspace{1cm}
\includegraphics[scale=0.35]{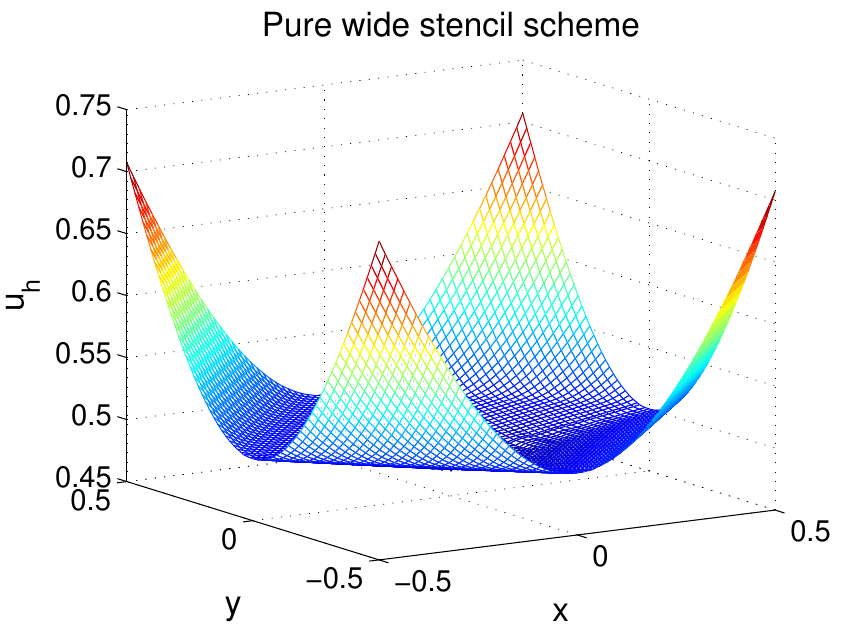}
\\
(1)
\hspace{5cm}
(2)
\\
\includegraphics[scale=0.4]{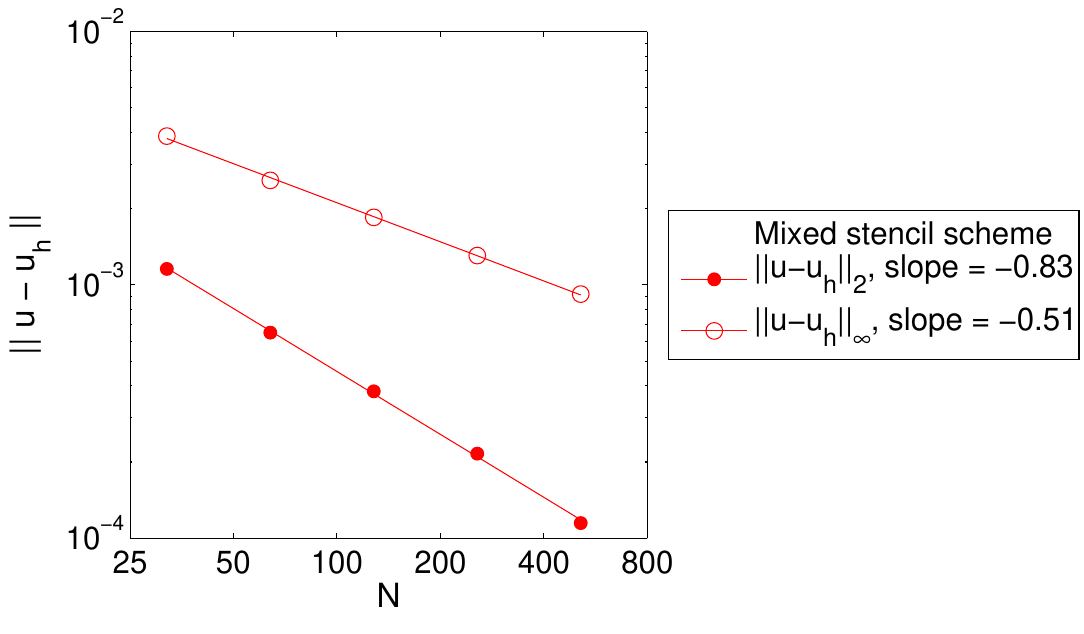}
\vspace{-2mm}
\\
(3)
\end{center}
\vspace{-2mm}
\caption{\label{fig:example4}
Numerical results of Example 4, where the exact solution is $u(x,y) = \sqrt{x^2+y^2}$.
(1) Numerical solution by the proposed mixed stencil scheme, which converges to the exact solution.
(2) Numerical solution by the pure semi-Lagrangian wide stencil scheme, which does not converge to the exact solution.
(3) Norms of the error $\|u-u_h\|$. The proposed mixed stencil scheme is used. The convergence rates, indicated by the slopes, are $O(h^{0.8})$ in $L_2$ norm and $O(h^{0.5})$ in $L_\infty$ norm, respectively.}
\end{figure}

\begin{table}[b!]
\footnotesize
\begin{center}
Proposed mixed stencil scheme
\begin{tabular}{|c||c|c||c|c||c|}
\hline
$N$
& $\|u-u_h\|_2$
& \begin{tabular}{c}
Numerical
\\
convergence
\\
rate
\end{tabular}
& $\|u-u_h\|_\infty$
& \begin{tabular}{c}
Numerical
\\
convergence
\\
rate
\end{tabular}
& \begin{tabular}{c}
Number of
\\
policy
\\
iterations
\end{tabular}
\\
\hline
%16
%& 1.886$\times 10^{-3}$
%& 
%& 5.209$\times 10^{-3}$
%& 
%& 8
%\\
32
& 1.156$\times 10^{-3}$
& %0.71
& 3.868$\times 10^{-3}$
& %0.42
& 9
\\
64
& 6.484$\times 10^{-4}$
& 0.83
& 2.583$\times 10^{-3}$
& 0.58
& 15
\\
128
& 3.803$\times 10^{-4}$
& 0.77
& 1.848$\times 10^{-3}$
& 0.48
& 17
\\
256
& 2.159$\times 10^{-4}$
& 0.82
& 1.305$\times 10^{-3}$
& 0.50
& 23
\\
512
& 1.148$\times 10^{-4}$
& 0.91
& 9.203$\times 10^{-4}$
& 0.50
& 27
\\
\hline
\end{tabular}
\end{center}
\smallskip
\caption{\label{tab:example4}
Numerical results of Example 4. The exact solution is $u(x,y) = \sqrt{x^2+y^2}$. The proposed mixed stencil scheme is used.}
\end{table}

\underline{Example 4.}
In practice, our numerical scheme can converge to not only viscosity solutions, but also a type of more general weak solutions, called Aleksandrov solutions \cite{gutierrez2012monge}. In this example, the corresponding $f$ is a delta function at the origin and is zero elsewhere:
\begin{equation*}
f(x,y) = \pi \delta(0,0),
\quad
g(x,y) = \sqrt{x^2+y^2},
\quad
\overline{\Omega} = [-0.5,0.5]\times[-0.5,0.5].
\end{equation*}
The exact solution
$
u(x,y)=\sqrt{x^2+y^2}
$
is an Aleksandrov solution. It is a $C^0$ function and is singular at the origin. Figure \ref{fig:example4}(1) shows that our proposed mixed scheme converges to the cone-shaped Aleksandrov solution. Conversely, Figure \ref{fig:example4}(2) shows that the pure semi-Lagrangian wide stencil scheme in \cite{feng2016convergent} does not give the cone-shaped Aleksandrov solution. Indeed, there is no theoretical proof that the pure wide stencil scheme can converge to Aleksandrov solutions. Figure \ref{fig:example4}(3) and Table \ref{tab:example4} report the convergence results by the proposed mixed scheme. The orders of convergence are close to 0.8 and 0.5 in $L_2$ and $L_\infty$ norms respectively.

\underline{Example 5.}
In order to make a case for designing a monotone numerical scheme that converges to the viscosity solution (which is convex), we show explicitly that non-monotone numerical scheme may converge to a non-viscosity solution (which may be non-convex). More analysis on this issue can be found in \cite{froese2011convergent,benamou2010two}. We consider
\begin{equation*}
f(x,y) = 1,
\qquad
g(x,y) = 0,
\qquad
\overline{\Omega} = [-0.5,0.5]\times[-0.5,0.5].
\end{equation*}
For this example, the exact solution $u$ is not smooth near $\partial\Omega$ \cite{benamou2010two}. Since a closed-form expression for $u$ is not available, we follow \cite{benamou2010two} and study the convergence behavior of $u_h$ towards $u$ by checking the values of $u_h(0,0)$ as $h\to 0$. The numerical solution using our monotone mixed scheme converges to the convex viscosity solution as $h\to 0$; see Figure \ref{fig:example5} and Table \ref{tab:example5}.
Alternatively, we consider a possible non-monotone discretization for $u_{xx}u_{yy}-u_{xy}^2=f$, which is the direct application of the standard central differencing on $u_{xx}$, $u_{yy}$ and the standard 4-point central differencing on $u_{xy}$. In our numerical experiment, the numerical solution under the non-monotone discretization converges to a concave function as $h\to 0$. We note that \cite{benamou2010two} has considered the same example using non-monotone discretization, and obtained another non-viscosity solution that is non-convex near $\partial\Omega$.

\begin{figure}[t!]
\begin{center}
\includegraphics[scale=0.35]{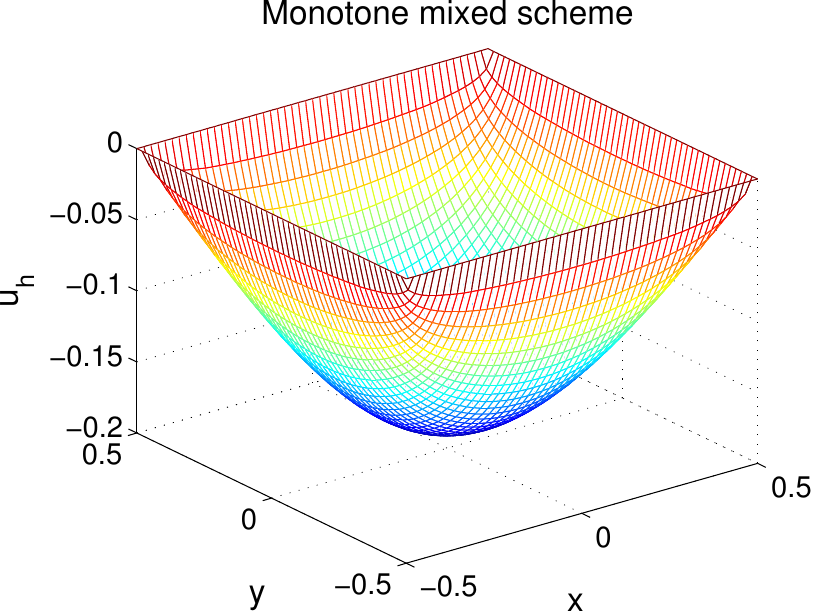}
\hspace{1cm}
\includegraphics[scale=0.35]{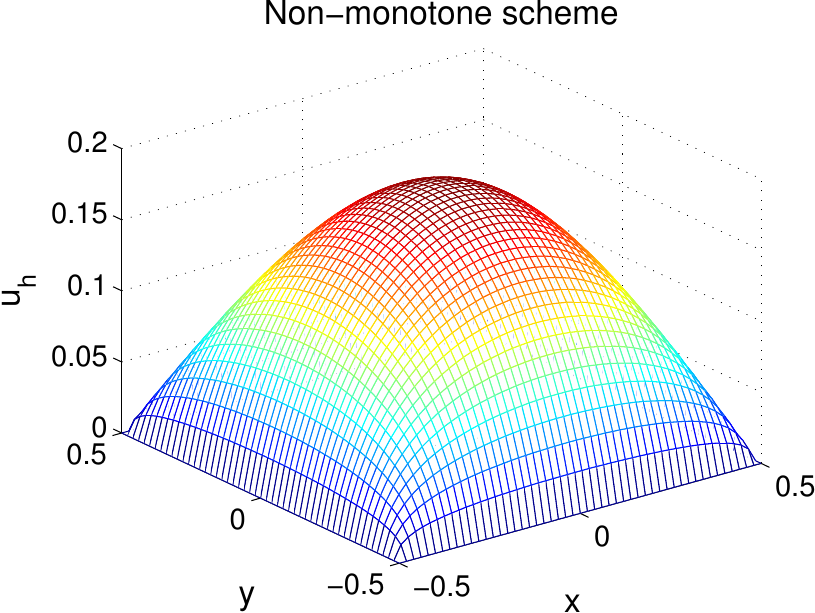}
\\
(1)
\hspace{5cm}
(2)
\end{center}
\caption{\label{fig:example5}
Example 5:
\,
(1) The solution given by the monotone mixed scheme, which is convex and is convergent in the viscosity sense.
\,
(2) One possible solution given by a non-monotone scheme, which is concave and is not a viscosity solution.
}
\end{figure}

\begin{table}[t!]
\footnotesize
\begin{center}
\begin{tabular}{|c|c|c|}
\hline
$N$
&
\begin{tabular}{c}
$u_h(0,0)$ by
\\
monotone scheme
\end{tabular}
&
\begin{tabular}{c}
$u_h(0,0)$ by
\\
non-monotone scheme
\end{tabular}
\\
\hline
32
& -0.18380
& 0.18063
\\
64
& -0.18444
& 0.18312
\\
128
& -0.18461
& 0.18436
\\
256
& -0.18485
& 0.18499
\\
512
& -0.18507
& 0.18530
\\
\hline
\end{tabular}
\end{center}
\smallskip
\caption{\label{tab:example5}
Example 5: (1) The minimum values of the numerical solutions $u_{min}$ given by the monotone mixed scheme, which provides an evidence that the numerical solution converges to a convex solution. (2) The maximum values of the numerical solutions $u_{max}$ given by a non-monotone scheme, which provides an evidence that the numerical solution converges to a non-convex solution. }
\end{table}

%% Conclusion
%% ==============

\section{Conclusion}
\label{sec:conclusion}

In this paper, we convert the Monge-Amp\`ere equation into the equivalent HJB equation, and propose a mixed finite difference discretization for solving the equivalent HJB equation. The discretization satisfies consistency, stability, monotonicity and strong comparison principle, and thus convergent to the viscosity solution of the Monge-Amp\`ere equation. Our proposed mixed scheme significantly improves the accuracy over the pure semi-Lagrangian scheme in \cite{feng2016convergent}. More specifically, the proposed mixed scheme yields a smaller discretization error $\|u - u_h\|$.
%We note that this is true even when semi-Lagrangian stencils are applied on part of the computational domain.
Furthermore, if the standard 7-point stencils can be applied on the entire computational domain monotonically, then our proposed mixed stencil scheme can improve the convergence rate to $O(h^2)$.

Our mixed scheme can be potentially extended to higher dimensional cases. Assuming that the dimension is $d$, the idea is to parametrize the control of the HJB equation (\ref{eq:convertHJB}), namely to parametrize $A(\mathbf{x})=Q(\mathbf{x})\Lambda(\mathbf{x}) Q(\mathbf{x})^T$, where $Q(\mathbf{x})\in SO(d)$ and $\Lambda(\mathbf{x})$ is a trace-1 non-negative diagonal matrix. Then the standard 7-point stencil discretization can be applied if $A(\mathbf{x})$ is weakly diagonal dominant, and the semi-Lagrangian wide stencil discretization is applied otherwise. We leave this topic as a future work.

\bibliographystyle{spmpsci.bst}
\bibliography{Reference}

\end{document}